\documentclass{amsart}%
\usepackage{amssymb,amsmath,xcolor,graphicx,xspace}
\usepackage[raggedrightboxes]{ragged2e}
\usepackage{amsfonts}
\usepackage{amsmath}
\usepackage{amssymb}
\usepackage{graphicx}
\usepackage[unicode=true]{hyperref}
\usepackage{ulem}
\usepackage{cite}%
\providecommand{\U}[1]{\protect\rule{.1in}{.1in}}
\providecommand{\U}[1]{\protect\rule{.1in}{.1in}}
\DeclareGraphicsExtensions{.pdf,.eps,.ps,.png,.jpg,.jpeg}
\providecommand{\U}[1]{\protect\rule{.1in}{.1in}}
\newtheorem{theorem}{Theorem}
\theoremstyle{plain}

\newtheorem{corollary}{Corollary}

\newtheorem{definition}{Definition}
\newtheorem{example}{Example}

\newtheorem{lemma}{Lemma}

\newtheorem{proposition}{Proposition}
\newtheorem{remark}{Remark}

\numberwithin{equation}{section}
\begin{document}
\title[Reverse H\"{o}lder inequalites revisited]{Reverse H\"{o}lder inequalites revisited: interpolation, extrapolation,
indices and doubling}
\author{Alvaro Corval\'{a}n}
\address{Instituto del Desarrollo Humano, Universidad Nacional de General Sarmiento, J.
M. Guti\'{e}rrez 1150, C.P. 1613, Malvinas Argentinas, Pcia de Bs.As,
Rep\'{u}blica Argentina}
\email{acorvala@ungs.edu.ar}
\author{Mario Milman}
\address{IAM- Instituto Argentino de Matem\'{a}tica "Alberto Calder\'{o}n", Buenos
Aires, Argentina}
\email{mario.milman@gmail.com}
\urladdr{https://sites.google.com/site/mariomilman}
\thanks{This paper is in final form and no version of it will be submitted for
publication elsewhere.}
\subjclass{ }

\begin{abstract}
Extending results in \cite{M} and \cite{MM} we characterize the classical
classes of weights that satisfy reverse H\"{o}lder inequalities in terms of
indices of suitable families of $K-$functionals of the weights. In particular,
we introduce a Samko type of index (cf. \cite{kara}) for families of
functions, that is based on quasi-monotonicity, and use it to provide an index
characterization of the $RH_{p}$ classes, as well as the limiting class $RH=$
$RH_{LLogL}=$. $\bigcup\limits_{p>1}RH_{p}$ (cf. \cite{BMR}),\ which in the
abstract case involves extrapolation spaces. Reverse H\"{o}lder inequalities
associated to $L(p,q)$ norms, and non-doubling measures are also treated.

\end{abstract}
\maketitle
\tableofcontents

\section{Introduction}

The usual applications of interpolation theory deal with the study of scales
of function spaces, and the operators acting on them. Indeed, the impact of
interpolation theory in classical analysis, pde's, approximation theory, and
functional analysis, is well documented (cf. \cite{BS}, \cite{BL},
\cite{Brud}, \cite{BB}, \cite{Krein}, \cite{ov}, \cite{tor1}, \cite{Tr}, and
the references therein). Somewhat less known is the fact that some of the
underlying techniques of interpolation theory can be also applied successfully
to study problems that one usually does not describe as \textquotedblleft
interpolation theoretic problems\textquotedblright.

In this vein, in \cite{M}, \cite{BMR}, \cite{MM}, \cite{MarM}, \cite{BMR1},
\cite{MarM1}, \cite{Mio}, we developed new methods to study classes of weights
that satisfy reverse H\"{o}lder inequalities, using tools from real
interpolation theory. It was shown how to transform the classical definitions
of the theory of reverse H\"{o}lder inequalities into inequalities for
suitable families of the $K-$functionals of the weights, that when combined
with the basic properties of the theory of real interpolation spaces, like the
representation of interpolation norms as averages of end-point norms
(\textquotedblleft reiteration\textquotedblright),\ with their crucial
\textquotedblleft scaling\textquotedblright, implied differential inequalities
whose solutions yield classical \textquotedblleft open
properties\textquotedblright\footnote{In a somewhat different direction, in
\cite{kru} other classical \textquotedblleft open\textquotedblright\ or self
improving properties, e.g. the open mapping theorem, were connected to a
suitable notion of distance for interpolation spaces. The precise relationship
between \cite{kru}, and the developments presented in this paper remains to be
investigated.}\ like Gehring's Lemma.

To better explain the contents of this paper it will be useful to review some
of the basic ideas connecting reverse H\"{o}lder inequalities, indices,
interpolation, and extrapolation\footnote{We have tried to accomodate
prospective readers that could be interested in the theory of weighted norm
inequalities or interpolation theory but perhaps are not familiar with both
areas simultaneously. This has led to a longer introduction, which we hope
will facilitate to introduce the underlying ideas to readers that feel that
they do not have the adequate background.}. We refer to Section \ref{sec:back}
for unexplained notation and background on interpolation theory and reverse
H\"{o}lder inequalities.

Given $1<p<\infty,$ we shall say that a weight\footnote{A positive locally
integrable function on $\mathbb{R}^{n}.$} $w$ satisfies a $p-$reverse
H\"{o}lder inequality, and we write $w\in RH_{p}$, if there exists a
constant\footnote{usually denoted by $\left\Vert w\right\Vert _{RH_{p}}$ (cf.
(\ref{revn}) below).} $C>0$ such that, for all cubes\footnote{In this paper
all the cubes are assumed to have their sides parallel to the coordinate
axes.} $Q$, we have%
\begin{equation}
\left\{  \frac{1}{\left\vert Q\right\vert }\int_{Q}w(x)^{p}dx\right\}
^{1/p}\leq C\left\{  \frac{1}{\left\vert Q\right\vert }\int_{Q}w(x)dx\right\}
. \label{intro1}%
\end{equation}
Fix a cube $Q_{0}.$ Through the use of local maximal inequalities, the fact
that $L^{p}=(L^{1},L^{\infty})_{1/p^{\prime},p}$, known computations of the
corresponding $K-$functionals (cf. (\ref{forK})), and the scaling provided by
Holmstedt's reiteration formula (cf. \cite[Corollary 3.6.2 (b), page 53]{BL}),%
\begin{equation}
K(t^{1/p},w;(L^{1},L^{\infty})_{1/p^{\prime},p},L^{\infty})\approx\left\{
\int_{0}^{t}[K(s,;L^{1},L^{\infty})s^{-1/p^{\prime}}]^{p}\frac{ds}{s}\right\}
^{1/p}, \label{intro3}%
\end{equation}
one can see (cf. \cite{M}) that (\ref{intro1}) implies that with a constant
independent of $Q_{0}$ we have that for all $0<t<\left\vert Q_{0}\right\vert
$,
\begin{align}
&  {\small K(t}^{1/p}{\small ,w\chi}_{Q_{0}}{\small ;(L}^{1}{\small (Q}%
_{0}{\small ),L}^{\infty}{\small (Q}_{0}{\small ))}_{1/p^{\prime}%
,p}{\small ,L}^{\infty}{\small (Q}_{0}{\small ))}\label{intro2}\\
&  {\small \leq}C{\small t}^{-1/p^{\prime}}{\small K(t,w\chi}_{Q_{0}%
}{\small ;L}^{1}{\small (Q}_{0}{\small ),L}^{\infty}{\small (Q}_{0}%
{\small )).}\nonumber
\end{align}
\newline Conversely, if there exists a constant $C>0$ such that (\ref{intro2})
holds for all cubes then it follows that $w\in RH_{p}$ (cf. Theorem
\ref{teo:p} in Section \ref{sec:classicalvs}). Moreover, underlying the
discussion above is the characterization of $RH_{p}$ through an implicit
differential inequality (cf. \cite{M}, \cite{BMR}, \cite{MM}). For a weight
$w\in RH_{p}$ and each cube $Q$, we let $\phi_{w,Q,1/p^{\prime}}%
(s)=K(s,w\chi_{Q};L^{1}(Q),L^{\infty}(Q))s^{-1/p^{\prime}},$ then there exists
a universal constant $C>0,$ such that for all cubes $Q,$
\begin{equation}
\int_{0}^{t}(\phi_{w,Q,1/p^{\prime}}(s))^{p}\frac{ds}{s}\leq C(\phi
_{w,Q,1/p^{\prime}}(t))^{p},0<t<\left\vert Q\right\vert . \label{bene1}%
\end{equation}

The inequality (\ref{bene1}) is central to our approach to reverse H\"{o}lder
inequalities (cf. \cite{M}, \cite{BMR}, \cite{MM}). Moreover, as it turns out,
the characterization of the solutions of inequalities of the form
(\ref{bene1}) is one of the achievements of all the classical theories of
indices (cf. \cite{BaSt}, \cite{boyd}, \cite{Krein}, \cite{Mal}, \cite{MM},
\cite{Sa}, and the references therein). Index theory shows that for each
\textbf{fixed} cube $Q_{0},$ we have the equivalence%
\begin{equation}
\int_{0}^{t}\left(  \phi_{w,Q_{0},1/p^{\prime}}(s)\right)  ^{p}\frac{ds}%
{s}\leq C\left(  \phi_{w,Q_{0},1/p^{\prime}}(t)\right)  ^{p}\Leftrightarrow
index(\phi_{w,Q_{0},1/p^{\prime}})>0, \label{sii}%
\end{equation}
where \textquotedblleft$index(\phi_{w,Q_{0},1/p^{\prime}})^{\prime\prime}$ is
a number, that can be defined in different ways (e.g. (e.g. \cite{Krein},
\cite{Mal}, \cite{Sa2})) and whose precise definition is not important right
now. However, for our purposes in this paper, we need to extend the
equivalence (\ref{sii}) in order to deal with \textbf{all} the cubes $Q_{0}$,
with a uniform constant $C.$ In other words, we need to extend the notion of
index originally defined on single functions to include families of functions.

In this paper we undertake to formalize some of the connections between
interpolation methods and the classical methods to study reverse H\"{o}lder
inequalities. In particular, we develop a suitable definition of
indices\footnote{. Our main inspiration for this came from \cite{M} that shows
that solutions of (\ref{bene1}) are quasi-increasing and the work of Samko and
her collaborators (cf. \cite{Sa}, \cite[Theorem 3.6]{kara}, \cite{Sa2}), who
among other definitions considers an index based on the notion of
quasi-monotonicity.} for the families of functions that allows us to extend
the equivalence (\ref{sii}) to the realm of families of functions. We define
the index of the family $\{K(\cdot,w\chi_{Q};L^{1}(Q),L^{\infty}(Q))\}_{Q}$
and obtain in the process a complete characterization of the reverse
H\"{o}lder classes of weights in terms of our indices (cf. Theorem
\ref{teo:galvanizado}):%
\begin{equation}
RH_{p}=\{w:ind\{K(\cdot,w\chi_{Q};L^{1}(Q),L^{\infty}(Q))\}_{Q}>1/p^{\prime
}\}. \label{rhp}%
\end{equation}
This characterization leads to a simple explanation of the open or self
improving properties underlying the theory (e.g. Gehring's Lemma). Indeed, if
$w\in RH_{p}$ then by (\ref{rhp}) it is possible to select $\varepsilon
:=\varepsilon(w),$ such that $p_{0}=p+\varepsilon,$ is such that
$ind\{K(\cdot,w\chi_{Q};L^{1}(Q),L^{\infty}(Q))\}_{Q}>1/p_{0}^{\prime
}>1/p^{\prime}$ and therefore, once again by (\ref{rhp}), it follows that
$w\in RH_{p_{0}}.$

The case $p=1$ requires a different treatment since in this case the
inequality (\ref{intro1}) is true for all weights$.$ Moreover, the usual form
of Holmstedt's formula (\ref{intro3}) does not hold. On the other hand, if we
replace $1/p^{\prime}$ by $0,$ then (\ref{bene1}) still makes sense and,
indeed, plays a r\^{o}le in the characterization of the limiting class $RH.$
It turns out that the description of $RH$ is connected with extrapolation
spaces (cf. \cite{BMR}, \cite{Go}, \cite{as}). We shall now develop this point
in some detail.

Let
\begin{equation}
RH=\bigcup\limits_{p>1}RH_{p}, \label{intro10}%
\end{equation}
then the limiting case of (\ref{rhp}) can be stated as (cf. Theorem
\ref{teo:galvanizado} (ii))%
\begin{equation}
RH=\{w:\text{ }ind\{K(\cdot,w\chi_{Q};L^{1}(Q),L^{\infty}(Q))\}_{Q}>0\}.
\label{intro10'}%
\end{equation}
The elements of $RH$ can be also characterized explicitly in terms of
comparisons of their averages that, in the abstract case, involves the use of
extrapolation spaces. Indeed, it turns out that the correct reverse H\"{o}lder
inequality in the limiting case is to compare the $LLogL$ averages of $w$ with
its $L^{1}$ averages\footnote{Interestingly, in \cite{BMR} we arrived first to
this formulation using interpolation/extrapolation. It is one instance where
interpolation was used as a discovery tool in classical analysis.} (cf.
\cite{FE}, \cite{FeKPi}, \cite{BMR}). The result can be stated as follows. Let
$RH_{LLogL}$ be the class of weights\footnote{see Definition \ref{def:log}}
$w$ such that there exists a constant $c>0,$ such that for all cubes $Q,$
\begin{equation}
\left\Vert w\right\Vert _{LLogL(Q,\frac{dx}{\left\vert Q\right\vert })}\leq
c\left\{  \frac{1}{\left\vert Q\right\vert }\int_{Q}w(x)dx\right\}  ,
\label{intro6}%
\end{equation}
then (cf. Theorem \ref{teo:galvanizado} (iii)),%
\begin{equation}
RH=RH_{LLogL}. \label{intro15}%
\end{equation}

In this framework, Gehring's Lemma for $RH_{LLogL}$ (cf. \cite{FE},
\cite{FeKPi}, \cite{BMR}) follows from the fact if $ind\{K(\cdot,w\chi
_{Q};L^{1}(Q),L^{\infty}(Q))\}_{Q}>0,$ then we can choose $p:p(w)>1,$ such
that (cf. \cite{MM}) $ind\{K(\cdot,w\chi_{Q};L^{1}(Q),L^{\infty}%
(Q))\}_{Q}>1/p^{\prime}$.

It is important to mention that the formalism we have outlined above works,
and indeed was first developed, in the general setting of
interpolation/extrapolation spaces. In the abstract theory we replace the pair
$(L^{1},L^{\infty})$ by a Banach pair $(X_{0},X_{1})$, and $(L^{1},L^{\infty
})_{1/p^{\prime},p}$ by $(X_{0},X_{1})_{\theta,q}$ (cf. Section \ref{sec:K},
Definition \ref{def:abajo}) and, of course, there are no considerations of
cubes. Note that, in general, the index \textquotedblleft$q$\textquotedblright%
\ may not be correlated in a specific way to the first index \textquotedblleft%
$\theta$\textquotedblright, and this uncoupling already manifests itself when
dealing with $L(p,q)$ spaces, as we now explain.

The Lorentz $L(p,q)=(L^{1},L^{\infty})_{1/p^{\prime},q}$ spaces are
quintessential interpolation spaces, so it is instructive to indicate some
possibly new results on reverse H\"{o}lder inequalities for Lorentz spaces,
that can be derived using our methods. For this purpose we now recall the
appropriate scaling that we use to define averages of Lorentz norms. It will
be actually easier to frame the discussion in a slightly more general setting.

Let $X:=X(R^{n})$ be a rearrangement invariant space, and let $X^{\prime}$ be
the associate space of $X.$ It is well known, and easy to see (cf. \cite{BS}),
that for every cube $Q,$ we have $\left\Vert \chi_{Q}\right\Vert
_{X}\left\Vert \chi_{Q}\right\Vert _{X^{\prime}}=\left\vert Q\right\vert ,$
this fact, combined with H\"{o}lder's inequality, yields
\[
\int\left\vert w\chi_{Q}\right\vert \leq\left\Vert w\chi_{Q}\right\Vert
_{X}\left\Vert \chi_{Q}\right\Vert _{X^{\prime}},
\]
yields%
\begin{equation}
\frac{1}{\left\vert Q\right\vert }\int\left\vert w\chi_{Q}\right\vert
\leq\frac{\left\Vert w\chi_{Q}\right\Vert _{X}}{\left\Vert \chi_{Q}\right\Vert
_{X}}. \label{intro12}%
\end{equation}
In this context the natural maximal operator is given by (cf. \cite{BMR2}, and
the references therein),
\[
M_{X}w(x)=\sup\limits_{Q\ni x}\frac{\left\Vert w\chi_{Q}\right\Vert _{X}%
}{\left\Vert \chi_{Q}\right\Vert _{X}}.
\]
It follows that,%
\begin{equation}
Mw(x)\leq M_{X}w(x), \label{intro13'}%
\end{equation}
where $M:=M_{L^{1}}$ is the maximal operator of Hardy-Littlewood. For example,
if $X=L^{p},$ $1\leq p<\infty,$ then $M_{L^{p}}f=\{M\left\vert f\right\vert
^{p}\}^{1/p}.$ The choice $X=L(p,q)$ is of particular interest, for in this
case $M_{X}$ corresponds to the maximal function introduced by Stein
\cite{St}, that plays a r\^{o}le in the theory of Sobolev spaces and other
areas of classical Analysis (cf. \cite{Mo} for a recent reference).

To reverse (\ref{intro12}), we introduce the class $RH_{X}$ of weights $w$
such that there exists a constant $c>0$, such for all cubes $Q,$ it holds%
\begin{equation}
\frac{\left\Vert w\chi_{Q}\right\Vert _{X}}{\left\Vert \chi_{Q}\right\Vert
_{X}}\leq c\frac{1}{\left\vert Q\right\vert }\int w\chi_{Q}. \label{intro11}%
\end{equation}
In terms of maximal functions, (\ref{intro11}) implies a reversal of
(\ref{intro13'}) which, in terms of rearrangements, is given by:%
\begin{equation}
(M_{X}w)^{\ast}(t)\leq c(Mw)^{\ast}(t). \label{intro13}%
\end{equation}
The class $RH_{p}$ corresponds to the choice $X=L^{p}.$ Moreover, note that
for $X=L(p,q),$ $\left\Vert \chi_{Q}\right\Vert _{L(p,q)}=\left\vert
Q\right\vert ^{1/p}$ and, therefore, the scaling of (\ref{intro11}) leads to
the consideration of averages controlled by $\frac{1}{t^{1/p}}K(t^{1/p}%
,f;L(p,q),L^{\infty})$ $.$ This scaling is compatible with the general
definition of reverse H\"{o}lder inequalities we give below (cf.
(\ref{Ktheta})) and, therefore, our interpolation machinery can be applied,
provided we have the appropriate rearrangement inequalities for the
corresponding maximal operators. Such inequalities are available for the
$L(p,q)$ spaces, with\footnote{Note that for $q\leq p,$ $RH_{L(p,q)}\subset
RH_{p},$ and therefore, $RH_{L(p,q)}$ inherits the self-improvement property
from $RH_{p}$.} $1<p\leq q.$ Indeed, in this case, the rearrangement
inequalities of \cite{BMR2} give%
\begin{equation}
(M_{L(p,q)}f)^{\ast}(t)\geq c\frac{1}{t^{1/p}}K(t^{1/p},f;L(p,q),L^{\infty}).
\label{intro14}%
\end{equation}
It follows that, if $w\in RH_{L(p,q)},$ $1<p\leq q,$ then, by (\ref{intro13})
and (\ref{intro14}), $w$ satisfies the $K-$functional inequality (cf.
(\ref{Ktheta}))
\[
\frac{1}{t^{1/p}}K(t^{1/p},f;L(p,q),L^{\infty})\leq C\frac{K(t,f;L^{1}%
,L^{\infty})}{t}.
\]
In other words, $w$ satisfies a version of (\ref{Ktheta}), and therefore (via
Holmstedt's formula!) we can setup an inequality of the form (\ref{sii}) and
show a Gehring self improving effect, even though the condition $w\in
RH_{L(p,q)},1<p\leq q,$ is weaker\footnote{since, indeed, by definition we
have $w\in RH_{p}\Rightarrow w\in RH_{L(p,q)}.$} than $w\in RH_{p}.$ However,
an argument with indices shows that the index $q$ is not important here (cf.
Remark \ref{introdespues}) and, in fact, we have,
\[
RH_{p}=RH_{L(p,q)},1<p\leq q<\infty.
\]
We refer to Section \ref{sec:lorentz} for a more detailed discussion.

The interpolation method can be also implemented when dealing with suitable
non-doubling weights that are absolutely continuous with respect to Lebesgue
measure (cf. \cite{MM}, \cite{MarM} for reverse H\"{o}lder inequalities, and
also \cite{Oro}, where the corresponding theory of  $A_{p}$ weights for
non-doubling measure is treated using classical methods). This example should
be of interest to classical analysts, and aficionados of interpolation theory.
The issue at hand is that in the non-doubling setting the leftmost inequality
in the following chain%
\[
(Mw)^{\ast}(t)\approx w^{\ast\ast}(t)=\frac{K(t,w;L^{1},L^{\infty})}{t}%
\]
does not hold (cf. \cite{aa}). Thus, a different mechanism is needed to relate
the information on the averages of $w,$ coming from conditions like $RH_{p},$
to information about $K-$functionals. The appropriate solution in this case is
to dispense with the classical maximal operator altogether and work directly
with a different expression of the $K-$functional for the weighted pair
$(L_{w}^{p}(\mathbb{R}^{n}),L^{\infty}(\mathbb{R}^{n}))$. Such formulae was
obtained in \cite{aa} (cf. also \cite{MarM} and Section \ref{sec:doubling}
below). We shall now review that part of the story.

For a given sequence of disjoint cubes (\textquotedblleft
packing\textquotedblright) $\pi=\left\{  Q_{i}\right\}  _{i=1}^{\left\vert
\pi\right\vert },$ with $\left\vert \pi\right\vert =\#$ cubes in $\pi$, we
associate a linear operator $S_{\pi}$, defined by
\begin{equation}
S_{\pi}(f)(x)=\sum_{i=1}^{\left\vert \pi\right\vert }\left(  \frac{1}%
{w(Q_{i})}\int_{Q_{i}}f(y)w(y)dy\right)  \chi_{Q_{i}}(x),\text{ \ \ \ \ }f\in
L_{w}^{1}(\mathbb{R}^{n})+L^{\infty}(\mathbb{R}^{n}), \label{spi}%
\end{equation}
and let $\left(  F_{\left\vert f\right\vert ^{p}}\right)  _{w}\,$\ be the
maximal operator defined by\footnote{Here $\ast_{w}$ are rearrangements with
respect to the measure $w(x)dx.$}%

\begin{equation}
\left(  F_{\left\vert f\right\vert ^{p}}\right)  _{w}(t)=\sup\limits_{\pi
}\left(  S_{\pi}(\left\vert f\right\vert ^{p})\right)  _{w}^{\ast}(t).
\label{nn}%
\end{equation}
Then,
\begin{equation}
K(t,f;L_{w}^{p}(\mathbb{R}^{n}),L^{\infty}(\mathbb{R}^{n}))\approx
t^{1/p}\left(  F_{\left\vert f\right\vert ^{p}}\right)  _{w}(t). \label{nn1}%
\end{equation}
Using this tool we can bypass the use of the classical maximal operator and
readily show that the reverse H\"{o}lder inequalities can be formulated as
$K-$functional inequalities of the form (\ref{Ktheta}) (cf. Section
\ref{sec:doubling}, (\ref{conpeso})), thus making available the interpolation
machinery, including the characterization of these classes of weights via
indices (cf. (\ref{rhp})).

In Section \ref{sec:Comparison} we consider other applications of our theory.
For example, the well known connection between weights that satisfy reverse
H\"{o}lder inequalities and the $A_{p}$ weights of Muckenhoupt, one of whose
manifestations is given by the equality $A_{\infty}=RH,$ which, combined with
(\ref{intro10'}), gives an index characterization of $A_{\infty},$%
\[
A_{\infty}=\{w:ind\{K(\cdot,w\chi_{Q};L^{1}(Q),L^{\infty}(Q))\}_{Q}>0\}.
\]
In Section \ref{sub:carro} we compare our result with the recent
characterization of $A_{\infty}$ obtained in \cite{ACKS} using different
indices and without use of interpolation methods. It is shown in \cite{ACKS}
that
\[
A_{\infty}=\{w:\widetilde{ind}(w)<1\},
\]
where $\widetilde{ind}(w)$ is an index introduced in \cite{ACKS} independently
from the theory of indices or interpolation theory. In Section \ref{sub:carro}%
, (\ref{acks2}), we give a direct proof of
\[
ind\{K(\cdot,w\chi_{Q};L^{1}(Q),L^{\infty}(Q))\}_{Q}>0\Leftrightarrow
\widetilde{ind}(w)<1,
\]
which clarifies the situation.

As another application of our methods, in Section \ref{sec:strom} we give a
simple proof of the important formula obtained by Stromberg-Wheeden (cf.
\cite{CN}),%
\[
A_{\infty}^{p}=RH_{p}.
\]
We have also included a small section of problems (cf. Section
\ref{sec:problems}) connected with the topics discussed in the paper.

The table of contents should serve as guide to the contents of the paper. A
few words about the bibliography are also in order. Documenting the material
discussed in the paper has resulted in a relatively large bibliography but,
alas, it was not our intention to compile a comprehensive one. We have not
attempted to cover the huge amount of material that falls outside our
development in this paper. Moreover, since interpolation methods up to this
point have not been mainstream in the theory of weighted norm inequalities,
and indeed one of the objectives of this paper is to help to try to reverse
this situation, our references tend to exhibit a distinctive vintage
character. Therefore, we apologize in advance if your favorite papers are not
quoted. We should also call attention to the fact that there is a literature
that utilizes some of the underlying technical tools that we use here but
implemented using a completely different point of view than ours. In
particular, without the use of interpolation theoretical methods we developed.
A case in point is our use of rearrangements, a technique that, indeed, goes
back to early papers on weighted norm inequalities (cf. \cite{Mu}) and has
been treated extensively by several authors (e.g. the Italian school (e.g.
\cite{Moscariello}) and others. )

\section{Background: Classes of weights and Interpolation
theory\label{sec:back}}

In this section we recall some basic definitions. Our main reference on
interpolation theory, function spaces and rearrangements will be \cite{BS}.
The references for theory of weighted norm inequalities we use are \cite{GR},
\cite{tor1}, and for Gehring's Lemma we refer to \cite{Iw}.

\subsection{Weights\label{sec:weights}}

We start by recalling the definition of $RH_{p}$, the class of weights that
satisfy a reverse H\"{o}lder inequality.

\begin{definition}
Let $1<p<\infty.$ A weight $w$ is a positive locally integrable function
defined on $R^{n}.$ We shall say that a weight $w$ belongs to the reverse
H\"{o}lder class $RH_{p},$ if there exists $C:=C(w)>0$ such that for all cubes
$Q\subset R^{n},$ we have:
\begin{equation}
\left(  \frac{1}{\left\vert Q\right\vert }\int_{Q}w\left(  x\right)
^{p}dx\right)  ^{\frac{1}{p}}\leq C\frac{1}{\left\vert Q\right\vert }\int%
_{Q}w\left(  x\right)  dx. \label{revq}%
\end{equation}
We let\footnote{By abuse of language we use the norm symbol here.}%
\begin{equation}
\left\Vert w\right\Vert _{RH_{p}}=\inf\{C:\text{(\ref{revq}) holds}\}.
\label{revn}%
\end{equation}
We shall also consider the limit class
\begin{equation}
RH=\bigcup\limits_{p>1}RH_{p}. \label{rh1}%
\end{equation}

\end{definition}

The $RH_{p}$ classes increase as $p\downarrow1,$ so one can ask whether $RH$
can be described by suitable limiting version of (\ref{revq}). In this regard
we note that, if we simply let $p=1$ in (\ref{revq}), the resulting condition
is satisfied by all weights. It turns out that the correct comparison
condition in the limiting case $p=1$ is to replace the $L^{p}(Q)$ averages by
$LLogL(Q)$ averages.

\begin{definition}
\label{def:log}(cf. \cite{FE}, \cite{FeKPi}, \cite{BMR}) We shall say that a
weight $w$ belongs to the reverse H\"{o}lder class $RH_{LLogL}$ if there
exists $C:=C(w)>0$ such that, for all cubes $Q$, we have%
\begin{equation}
\left\Vert w\right\Vert _{L(LogL)(Q,\frac{dx}{\left\vert Q\right\vert })}\leq
C\frac{1}{\left\vert Q\right\vert }\int_{Q}w\left(  x\right)  dx,
\label{revlog}%
\end{equation}
where%
\[
\left\Vert f\right\Vert _{L(LogL)(Q,\frac{dx}{\left\vert Q\right\vert })}%
=\inf\{r:\frac{1}{\left\vert Q\right\vert }\int_{Q}\frac{\left\vert
f(y)\right\vert }{r}\log(e+\frac{\left\vert f(y)\right\vert }{r})dy\leq1\},
\]
and we let%
\begin{equation}
\left\Vert w\right\Vert _{RH_{LLogL}}=\inf\{C:\text{(\ref{revlog}) holds}\}.
\label{revlogn}%
\end{equation}

\end{definition}

It is shown in the references mentioned above that (\ref{revlog}) is, indeed,
the correct condition to describe the limiting class of weights that satisfy a
reverse H\"{o}lder inequality (cf. Section \ref{RH_and_IND copy(1)}, Theorem
\ref{teo:limite} below):%
\[
RH=RH_{LLogL}.
\]

The reverse H\"{o}lder classes are connected with the Muckenhoupt $A_{p}$
classes of weights. On some occasion we shall refer to the connection between
these classes of weights, so we now briefly recall the definitions.

\begin{definition}
Let $p\in(1,\infty).$ We shall say that a weight $w$ belongs to the
(Muckenhoupt) class $A_{p},$ if there exists $C:=C(w)>0$ such that, for all
cubes $Q,$ it holds%
\begin{equation}
\left(  \dfrac{1}{\left\vert Q\right\vert }%
{\displaystyle\int\limits_{Q}}
w\left(  x\right)  dx\right)  \left(  \dfrac{1}{\left\vert Q\right\vert }%
{\displaystyle\int\limits_{Q}}
w\left(  x\right)  ^{\tfrac{-1}{p-1}}dx\right)  ^{p-1}\leq C. \label{Ap}%
\end{equation}

We shall say that $w$ belongs to $A_{1},$ if there exists $C:=C(w)>0$ such
that for every cube $Q:$
\begin{equation}
\frac{1}{\left\vert Q\right\vert }\int_{Q}w\left(  y\right)  dy\leq Cw(x),
\label{A1}%
\end{equation}
for almost every $x\in Q$. Equivalently, $w\in A_{1}$ iff there exists
$C:=C(w)>0$ such that
\begin{equation}
Mw(x)\leq Cw(x),a.e., \label{A1_bis}%
\end{equation}
where $M$ is the Hardy-Littlewood maximal operator defined by
\begin{equation}
Mw\left(  x\right)  =\sup\limits_{Q\ni x}\frac{1}{\left\vert Q\right\vert
}\int_{Q}\left\vert w\left(  s\right)  \right\vert ds. \label{maxim}%
\end{equation}
The corresponding limiting class of weights $A_{\infty}$ is defined by
\begin{equation}
A_{\infty}=\bigcup\limits_{p\geq1}A_{p}. \label{apinfty}%
\end{equation}

\end{definition}

One basic connection between the $RH_{p}$ and $A_{p}$ classes of weights is
given by the following well known limiting case identity (cf. \cite{CF})
\[
RH=A_{\infty}.
\]

Summarizing the results for the limiting classes of weights discussed, we have%
\[
RH_{LLogL}=RH=A_{\infty}.
\]

The $A_{p}$ and $RH_{p}$ classes enjoy the following well known
self-improvement property that can be described informally\footnote{In fact,
the self improvement of these classes is interconnected since, if we let
$RH_{p^{\prime}}(w(x)dx)$ denote the class of weights that belong to
$RH_{p^{\prime}}$ with respect to the measure $w(x)dx,$ then, as is well known
and easy to see, we have $w\in A_{p}\Leftrightarrow w^{-1}\in RH_{p^{\prime}%
}(w(x)dx)$ (cf. \cite{BMR} and the references therein).} as follows:%
\[
``w\in A_{p}\Rightarrow w\in A_{p-\varepsilon}\text{\textquotedblright\ and
\textquotedblleft}w\in RH_{p}\Rightarrow w\in RH_{p+\varepsilon}%
\text{\textquotedblright,}%
\]
where the \textquotedblleft$\varepsilon$\textquotedblright\ depends on $w.$ In
Section \ref{RH_and_IND copy(1)} we show how the indices introduced in this
paper give simple proofs of results in \cite{MM}, \cite{M} that in particular,
show that these \textquotedblleft open\textquotedblright\ \ or self-improving
properties admit a very simple interpretation in the abstract setting of
interpolation theory.

\subsection{Real interpolation and K-functionals\label{sec:K}}

Given a compatible pair of Banach spaces\footnote{This means that there exists
a topological vector space $V,$ such that $X_{0},X_{1}$ are continuously
embedded in $V$.}, $\vec{X}=\left(  X_{0},X_{1}\right)  ,$ the $K-functional$
of an element $w\in X_{0}+X_{1},$ is the nonnegative concave function on
$\mathbb{R}_{+}$ defined by\footnote{When there cannot be confusion we shall
simply write $K\left(  t,w\right)  $.}
\begin{equation}
K\left(  t,w;\vec{X}\right)  =\inf\limits_{w=x_{0}+x_{1}}\left\{  \left\Vert
x_{1}\right\Vert _{X_{0}}+t\left\Vert x_{0}\right\Vert _{X_{1}}\right\}  ,t>0.
\label{forK}%
\end{equation}
The interpolation spaces $\vec{X}_{\theta,q},\theta\in(0,1),1\leq q\leq
\infty,$ are defined by%
\[
\vec{X}_{\theta,q}=\{w\in X_{0}+X_{1}:\left\Vert w\right\Vert _{\vec
{X}_{\theta,q}}=\left\{  \int_{0}^{\infty}[K(s,w;\vec{X})s^{-\theta}%
]^{q}ds\right\}  ^{1/q}<\infty\},
\]
with the usual modification when $q=\infty.$

We shall say that a Banach pair $\vec{X}=\left(  X_{0},X_{1}\right)  $ is
\textquotedblleft ordered" if $X_{1}\subset X_{0},$ in which case we let
$n:=n_{X_{0},X_{1}}=\sup\limits_{f\in X_{1}}\frac{\left\Vert f\right\Vert
_{X_{0}}}{\left\Vert f\right\Vert _{X_{1}}},$ be the norm of the corresponding embedding.

\begin{definition}
Let $\vec{X}=(X_{0},X_{1})$ be an ordered pair. Then we let\footnote{Note that
for an ordered pair $\vec{X},$ the $K-$functional is constant for $t>n.$ Also
note that in general $\int_{0}^{\infty}K(s,f;\vec{X})\frac{ds}{s}<\infty$
implies $f=0.$ For more on this we refer to \cite{Go}, \cite{as}.} $\vec
{X}_{0,1}=\{x:\left\Vert x\right\Vert _{\vec{X}_{0,1}}=\int_{0}^{n}%
K(s,x;\vec{X})\frac{ds}{s}<\infty\}.$
\end{definition}

The $\vec{X}_{0,1}$ spaces appear naturally in extrapolation theory (cf.
\cite{Go}, \cite{as}). Their import for our development here comes from the following

\begin{example}
\label{ejemplomarkao}(cf. \cite{Go}) Let $Q$ be a cube on $R^{n},$ and let
$\vec{X}=(L^{1}(Q),L^{\infty}(Q)).$ Then, $(L^{1}(Q),L^{\infty}(Q))$ is an
ordered pair and%
\begin{equation}
\vec{X}_{0,1}(Q)=(L^{1}(Q),L^{\infty}(Q))_{0,1}=\{f:\int_{0}^{\left\vert
Q\right\vert }f^{\ast}(s)\log\frac{\left\vert Q\right\vert }{s}ds<\infty\}.
\label{X01_cond2}%
\end{equation}

\end{example}

\begin{proof}
Since $\left\Vert f\right\Vert _{L^{1}(Q)}\leq\left\vert Q\right\vert
\left\Vert f\right\Vert _{L^{\infty}(Q)},$ $n=\left\vert Q\right\vert ,$ and
$K(t,f;L^{1}(Q),L^{\infty}(Q))=\int_{0}^{t}f^{\ast}(s)ds=tf^{\ast\ast}(t)$
becomes constant when $t>\left\vert Q\right\vert .$ Integration by parts
yields%
\begin{align*}
\int_{0}^{n}K(s,f;L^{1}(Q),L^{\infty}(Q))\frac{ds}{s}  &  =\int_{0}%
^{\left\vert Q\right\vert }sf^{\ast\ast}(s)\frac{ds}{s}\\
&  =\int_{0}^{\left\vert Q\right\vert }f^{\ast}(s)\log\frac{\left\vert
Q\right\vert }{s}ds,
\end{align*}
as we wished to show.
\end{proof}

In this abstract context we define \textquotedblleft reverse H\"{o}lder
classes\textquotedblright\ as follows

\begin{definition}
\label{def:abajo}(cf. \cite{M}) \ Let $\theta\in(0,1),1\leq q<\infty.$ Given a
Banach pair $\vec{X},$ we let $RH_{\theta,q}(\vec{X})$ be the class of
elements $w\in X_{0}+X_{1}$ such that there exists a constant $C=C_{w}(\vec
{X})>0,$ such that
\begin{equation}
K(t,w;\vec{X}_{\theta,q},X_{1})\leq Ct\frac{K(t^{\frac{1}{1-\theta}},w;\vec
{X})}{t^{\frac{1}{1-\theta}}},\text{ for all }t>0. \label{Ktheta}%
\end{equation}
We let%
\begin{equation}
\left\Vert w\right\Vert _{RH_{\theta,q}(\vec{X})}=\inf\{C:(\ref{Ktheta})\text{
holds}\}. \label{Kthetan}%
\end{equation}
Moreover, we let%
\begin{equation}
RH(\vec{X}):=\bigcup\limits_{(\theta,q)\in(0,1)\times\lbrack1,\infty
)}RH_{\theta,q}(\vec{X}). \label{rhabs}%
\end{equation}

\end{definition}

The corresponding limiting class $RH_{0,1}$ is given by

\begin{definition}
\label{interpola}(cf. \cite{BMR}) Let $\vec{X}$ be an ordered pair. We shall
say that $w\in X_{0}$ belongs to the class $RH_{0,1}(\vec{X})$ if there exists
$C:=C_{w}(\vec{X})>0$, such that, for all $0<t<n,$ it holds
\begin{equation}
\int_{0}^{t}K(s,w;\vec{X})\frac{ds}{s}\leq CK(t,w;\vec{X}). \label{Kcero}%
\end{equation}
We let%
\[
\left\Vert w\right\Vert _{RH_{0,1}(\vec{X})}=\inf\{C:(\ref{Kcero})\text{
holds}\}.
\]

\end{definition}

\begin{remark}
It is of interest to point out the connection of (\ref{Kcero}) with a limiting
form of Holmstedt's formulae. In fact, recall that using Holmstedt's formula
we can rewrite the inequality defining $RH_{\theta,1}(\vec{X}),$
\[
K(t,w,\vec{X}_{\theta,1},X_{1})\leq Ct\frac{K(t^{\frac{1}{1-\theta}},w;\vec
{X})}{t^{\frac{1}{1-\theta}}},
\]
as%
\[
\int_{0}^{t^{\frac{1}{1-\theta}}}s^{-\theta}K(s,w;\vec{X})\frac{ds}{s}\leq
Ct\frac{K(t^{\frac{1}{1-\theta}},w,\vec{X})}{t^{\frac{1}{1-\theta}}}.
\]
But since $s^{-\theta}$ decreases, the last inequality implies%
\[
t^{\frac{-\theta}{1-\theta}}\int_{0}^{t^{\frac{1}{1-\theta}}}K(s,w;\vec
{X})\frac{ds}{s}\leq Ct\frac{K(t^{\frac{1}{1-\theta}},w,\vec{X})}{t^{\frac
{1}{1-\theta}}}%
\]
yielding%
\[
\int_{0}^{t^{\frac{1}{1-\theta}}}K(s,w;\vec{X})\frac{ds}{s}\leq CK(t^{\frac
{1}{1-\theta}},w,\vec{X}),
\]
therefore if we formally let $\theta=0$ we obtain (\ref{Kcero}).
\end{remark}

The connection between the generalized reverse H\"{o}lder inequalities and the
classical definitions that were given in Section \ref{sec:weights} is
explained in the next section.

\section{Reverse H\"{o}lder inequalities: Classical vs Interpolation
definitions\label{sec:classicalvs}}

In this section we show the precise connection between the class of weights
that satisfy the classical reverse H\"{o}lder inequalities and the
corresponding definitions provided by interpolation theory.

\begin{theorem}
(cf. \cite{M}) \label{teo:p}Let $p>1$. Then, $w\in RH_{p}$ if and only for all
cubes $Q,$ $w\chi_{Q}$, the restriction of $w$ to $Q$, belongs to
$RH_{1-1/p,p}(L^{1}(Q),L^{\infty}(Q))$, and
\[
\sup\limits_{Q}\left\Vert w\chi_{Q}\right\Vert _{RH_{1-1/p,p}(L^{1}%
(Q),L^{\infty}(Q))}<\infty.
\]

\end{theorem}

\begin{proof}
Suppose that $w\in RH_{p},$ then for all cubes $Q,$%
\begin{equation}
\left(  \frac{1}{\left\vert Q\right\vert }\int_{Q}w\left(  x\right)
^{p}dx\right)  ^{\frac{1}{p}}\leq\left\Vert w\right\Vert _{RH_{p}}\frac
{1}{\left\vert Q\right\vert }\int_{Q}w\left(  x\right)  dx. \label{abs7}%
\end{equation}
Fix a cube $Q_{0}.$ Then, for all $x\in Q_{0}$ we have the pointwise
inequality,%
\begin{align*}
M_{p,Q_{0}}(w\chi_{Q_{0}})(x)  &  =\sup\limits_{Q\ni x,Q\subset Q_{0}}\left(
\frac{1}{\left\vert Q\right\vert }\int_{Q}(w\chi_{Q_{0}}(x))^{p}dx\right)
^{\frac{1}{p}}\\
&  \leq\left\Vert w\right\Vert _{RH_{p}}\sup\limits_{Q\ni x,Q\subset Q_{0}%
}\frac{1}{\left\vert Q\right\vert }\int_{Q}w\chi_{Q_{0}}(x)dx\\
&  =\left\Vert w\right\Vert _{RH_{p}}M_{Q_{0}}(w\chi_{Q_{0}})(x).
\end{align*}
By the well known Herz rearrangement inequalities (cf. \cite[Theorem 3.8, pag
122]{BS}) applied to $M_{Q_{0}},$ we have that, for $0<t<\left\vert
Q_{0}\right\vert $, and with absolute constants independent of $w,$ and
$Q_{0},$%
\begin{align*}
(M_{p,Q_{0}}(w\chi_{Q_{0}}))^{\ast}(t)  &  \approx\left\{  \frac{1}{t}\int%
_{0}^{t}(w\chi_{Q_{0}})^{\ast}(s)^{p}ds\right\}  ^{1/p},\text{ }\\
(M_{Q_{0}}(w\chi_{Q_{0}}))^{\ast}(t)  &  \approx\frac{1}{t}\int_{0}^{t}%
(w\chi_{Q_{0}})^{\ast}(s)ds.
\end{align*}
It follows that there exists a universal constant $C,$ independent of $w$ and
$Q_{0},$ such that, for all $0<t<\left\vert Q_{0}\right\vert ,$ we have%
\begin{equation}
\left\{  \frac{1}{t}\int_{0}^{t}(w\chi_{Q_{0}})^{\ast}(s)^{p}ds\right\}
^{1/p}\leq C\left\Vert w\right\Vert _{RH_{p}}\frac{1}{t}\int_{0}^{t}%
(w\chi_{Q_{0}})^{\ast}(s)ds.\text{ } \label{abs2}%
\end{equation}

On the other hand, since $L^{p}(Q_{0})=(L^{1}(Q_{0}),L^{\infty}(Q_{0}%
))_{1-1/p,p}$ and the $K-$ functional for the pair $(L^{p}(Q_{0}),L^{\infty
}(Q_{0})),1\leq p<\infty,$ is given by\footnote{The equivalence holds with
constants independent of $w\chi_{Q_{0}}$.}%
\begin{equation}
K(t^{1/p},w\chi_{Q_{0}};L^{p}(Q_{0}),L^{\infty}(Q_{0}))\approx\left\{
\int_{0}^{t}(w\chi_{Q_{0}})^{\ast}(s)^{p}ds\right\}  ^{1/p}, \label{abs8}%
\end{equation}
we can rewrite (\ref{abs2}) as follows: for all $0<t<\left\vert Q_{0}%
\right\vert ,$ we have%
\[
K(t^{1/p},w\chi_{Q_{0}};(L^{1}(Q_{0}),L^{\infty}(Q_{0}))_{1-1/p,p},L^{\infty
}(Q_{0}))
\]%
\begin{equation}
\leq\tilde{C}\left\Vert w\right\Vert _{RH_{p}}t^{-(1/p-1)}K(t,w\chi_{Q_{0}%
},L^{1}(Q_{0}),L^{\infty}(Q_{0})). \label{abs3}%
\end{equation}

Moreover, since the cube $Q_{0}$ was arbitrary, and the constant in
(\ref{abs3}) does not depend on $Q_{0},$ we conclude that $w\chi_{Q}\in
RH_{1-1/p,p}(L^{1}(Q),L^{\infty}(Q))$ for all cubes $Q$ and, moreover,%
\[
\sup_{Q}\left\Vert w\chi_{Q}\right\Vert _{RH_{1-1/p,p}(L^{1}(Q),L^{\infty
}(Q))}\leq\tilde{C}\left\Vert w\right\Vert _{RH_{p}},
\]
as we wished to show.

Conversely, suppose that $\sup\limits_{Q}\left\Vert w\chi_{Q}\right\Vert
_{RH_{1-1/p,p}(L^{1}(Q),L^{\infty}(Q))}<\infty.$ Fix a cube $Q_{0}$. Then, for
all $t>0$,
\begin{align*}
&  K(t^{p},w\chi_{Q_{0}};(L^{1}(Q_{0}),L^{\infty}(Q_{0}))_{1-1/p,p},L^{\infty
}(Q_{0}))\\
&  \leq\left(  \sup_{Q}\left\Vert w\chi_{Q}\right\Vert _{RH_{1-1/p,p}%
(L^{1}(Q),L^{\infty}(Q))}\right)  t^{-(1/p-1)}K(t,w\chi_{Q_{0}},L^{1}%
(Q_{0}),L^{\infty}(Q_{0})).
\end{align*}
Now, let $t=\left\vert Q_{0}\right\vert $ and use the identification
(\ref{abs8}) to obtain that for some absolute constant $\widetilde{C}$ not
depending on $Q_{0}$, it holds%
\begin{align*}
&  \left\{  \frac{1}{\left\vert Q_{0}\right\vert }\int_{0}^{\left\vert
Q_{0}\right\vert }(w\chi_{Q_{0}})^{\ast}(s)^{p}ds\right\}  ^{1/p}\\
&  \leq\tilde{C}\left(  \sup_{Q}\left\Vert w\chi_{Q}\right\Vert _{RH_{1-1/p,p}%
(L^{1}(Q),L^{\infty}(Q))}\right)  \frac{1}{\left\vert Q_{0}\right\vert }%
\int_{0}^{\left\vert Q_{0}\right\vert }(w\chi_{Q_{0}})^{\ast}(s)ds.
\end{align*}
Whence,%
\begin{align}
&  \left(  \frac{1}{\left\vert Q_{0}\right\vert }\int_{Q_{0}}w(x)^{p}%
dx\right)  ^{\frac{1}{p}}\label{abs10}\\
&  \leq\tilde{C}\left(  \sup_{Q}\left\Vert w\chi_{Q}\right\Vert _{RH_{1-1/p,p}%
(L^{1}(Q),L^{\infty}(Q))}\right)  \frac{1}{\left\vert Q_{0}\right\vert }%
\int_{Q_{0}}w\left(  x\right)  dx.\nonumber
\end{align}
Consequently, since $Q_{0}$ was arbitrary,%
\[
\left\Vert w\right\Vert _{RH_{p}}\leq\tilde{C}\left(  \sup_{Q}\left\Vert
w\chi_{Q}\right\Vert _{RH_{1-1/p,p}(L^{1}(Q),L^{\infty}(Q))}\right)
\]
as we wished to show.
\end{proof}

\begin{remark}
It follows from the proof that, with constants possibly depending
on$1<p<\infty,$ we have
\begin{equation}
\sup_{Q}\left\Vert w\chi_{Q}\right\Vert _{RH_{1-1/p,p}(L^{1}(Q),L^{\infty
}(Q))}\approx\left\Vert w\right\Vert _{RH_{p}}. \label{abs9}%
\end{equation}

\end{remark}

\subsection{The limiting case $p=1$}

In order to extend the results of the previous section to the limiting case
$p=1,$ and relate $RH_{0,1}$ to the condition provided by Definition
\ref{def:log}, we shall need to compare different norms for the space $LLogL.$
While such norm comparison results are part of the folklore, it is hard to
find references that provide a complete treatment that serves our
requirements, therefore, for the sake of completeness, we chose to provide
full details in the next lemma,

\begin{lemma}
\label{lema:previo}Suppose that $f\in LLogL_{loc}(\mathbb{R}^{n}).$ Then,

(i) For all cubes $Q$%
\begin{align*}
\frac{1}{\left\vert Q\right\vert }\int_{Q}\left\vert f(y)\right\vert
\log(e+\frac{\left\vert f(y)\right\vert }{\left\Vert f\chi_{Q}\right\Vert
_{L^{1}(Q,\frac{dx}{\left\vert Q\right\vert })}})dy  &  \leq2\left\Vert
f\right\Vert _{LLogL(Q,\frac{dx}{\left\vert Q\right\vert })}\\
&  \leq\frac{1}{\left\vert Q\right\vert }\int_{Q}\left\vert f(y)\right\vert
\log(e+\frac{\left\vert f(y)\right\vert }{\left\Vert f\chi_{Q}\right\Vert
_{L^{1}(Q,\frac{dx}{\left\vert Q\right\vert })}})dy,
\end{align*}
where $\left\Vert f\right\Vert _{L(LogL(Q,\frac{dx}{\left\vert Q\right\vert
})}$ denotes the $LLogL(Q,\frac{dx}{\left\vert Q\right\vert })$ Luxemburg norm
of $f,$%
\begin{equation}
\left\Vert f\right\Vert _{L(LogL(Q,\frac{dx}{\left\vert Q\right\vert }))}%
=\inf\{r:\frac{1}{\left\vert Q\right\vert }\int_{Q}\frac{\left\vert
f(y)\right\vert }{r}\log(e+\frac{\left\vert f(y)\right\vert }{r})dy\leq1\}.
\label{lux}%
\end{equation}

(ii) There exists an absolute constant such that for all cubes $Q$%
\begin{align*}
\frac{1}{\left\vert Q\right\vert }\int_{Q}\left\vert f(y)\right\vert
\log(e+\frac{\left\vert f(y)\right\vert }{\left\Vert f\chi_{Q}\right\Vert
_{L^{1}(Q,\frac{dx}{\left\vert Q\right\vert })}})dy  &  \leq\frac
{1}{\left\vert Q\right\vert }\int_{0}^{\left\vert Q\right\vert }(f\chi
_{Q})^{\ast}(s)\log(e+\frac{\left\vert Q\right\vert }{s})ds\\
&  \leq c\left\Vert f\chi_{Q}\right\Vert _{LLogL(Q,\frac{dx}{\left\vert
Q\right\vert })}.
\end{align*}

\end{lemma}

\begin{proof}
\textbf{(i) }Since the Young's function $y\log(e+y)$ satisfies the $\Delta
_{2}$ condition, the infimum in (\ref{lux}) is attained, and we have
\begin{equation}
\frac{1}{\left\vert Q\right\vert }\int_{Q}\left\vert f(y)\right\vert
\log(e+\frac{\left\vert f(y)\right\vert }{\left\Vert f\chi_{Q}\right\Vert
_{LLogL(Q,\frac{dx}{\left\vert Q\right\vert })}})dy=\left\Vert f\right\Vert
_{LLogL(Q,\frac{dx}{\left\vert Q\right\vert })}.\label{lux2}%
\end{equation}
In particular, since $\log(e+\frac{\left\vert f(y)\right\vert }{\left\Vert
f\chi_{Q}\right\Vert _{L(LogL(Q,\frac{dx}{\left\vert Q\right\vert })}})\geq1,$
we recover the well known fact that%
\begin{equation}
\left\Vert f\chi_{Q}\right\Vert _{L^{1}(Q,\frac{dx}{\left\vert Q\right\vert
})}=\frac{1}{\left\vert Q\right\vert }\int_{Q}\left\vert f(y)\right\vert
dy\leq\left\Vert f\right\Vert _{LLogL(Q,\frac{dx}{\left\vert Q\right\vert }%
)}.\label{lux1}%
\end{equation}
By (\ref{lux1}), $\frac{\left\Vert f\right\Vert _{L(LogL(Q,\frac
{dx}{\left\vert Q\right\vert })}}{\left\Vert f\chi_{Q}\right\Vert
_{L^{1}(Q,\frac{dx}{\left\vert Q\right\vert })}}\geq1$ and therefore we can
write,%
\begin{align*}
e+\frac{\left\vert f(y)\right\vert }{\left\Vert f\chi_{Q}\right\Vert
_{L^{1}(Q,\frac{dx}{\left\vert Q\right\vert })}} &  =e+\frac{\left\vert
f(y)\right\vert }{\left\Vert f\chi_{Q}\right\Vert _{L^{1}(Q,\frac
{dx}{\left\vert Q\right\vert })}}\frac{\left\Vert f\chi_{Q}\right\Vert
_{LLogL(Q,\frac{dx}{\left\vert Q\right\vert })}}{\left\Vert f\chi
_{Q}\right\Vert _{LLogL(Q,\frac{dx}{\left\vert Q\right\vert })}}\\
&  \leq(e+\frac{\left\vert f(y)\right\vert }{\left\Vert f\chi_{Q}\right\Vert
_{LLogL(Q,\frac{dx}{\left\vert Q\right\vert })}})\frac{\left\Vert f\chi
_{Q}\right\Vert _{LLogL(Q,\frac{dx}{\left\vert Q\right\vert })}}{\left\Vert
f\chi_{Q}\right\Vert _{L^{1}(Q,\frac{dx}{\left\vert Q\right\vert })}}.
\end{align*}
Consequently,%
\begin{align*}
&  \frac{1}{\left\vert Q\right\vert }\int_{Q}\left\vert f(y)\right\vert
\log(e+\frac{\left\vert f(y)\right\vert }{\left\Vert f\chi_{Q}\right\Vert
_{L^{1}(Q,\frac{dx}{\left\vert Q\right\vert })}})dy\\
&  \leq\frac{1}{\left\vert Q\right\vert }\int_{Q}\left\vert f(y)\right\vert
\log\left\{  (e+\frac{\left\vert f(y)\right\vert }{\left\Vert f\right\Vert
_{LLogL(Q,\frac{dx}{\left\vert Q\right\vert })}})\frac{\left\Vert f\right\Vert
_{LLogL(Q,\frac{dx}{\left\vert Q\right\vert })}}{\left\Vert f\chi
_{Q}\right\Vert _{L^{1}(Q,\frac{dx}{\left\vert Q\right\vert })}}\right\}  dy\\
&  =(I)+(II),
\end{align*}
where%
\begin{align*}
(I) &  =\frac{1}{\left\vert Q\right\vert }\int_{Q}\left\vert f(y)\right\vert
\log\left(  e+\frac{\left\vert f(y)\right\vert }{\left\Vert f\chi
_{Q}\right\Vert _{LLogL(Q,\frac{dx}{\left\vert Q\right\vert })}}\right)
dy=\left\Vert f\chi_{Q}\right\Vert _{LLogL(Q,\frac{dx}{\left\vert Q\right\vert
})}\text{ (by (\ref{lux2}))}\\
(II) &  =\log\left(  \frac{\left\Vert f\chi_{Q}\right\Vert _{LLogL(Q,\frac
{dx}{\left\vert Q\right\vert })}}{\left\Vert f\chi_{Q}\right\Vert
_{L^{1}(Q,\frac{dx}{\left\vert Q\right\vert })}}\right)  \frac{1}{\left\vert
Q\right\vert }\int_{Q}\left\vert f(y)\right\vert dy\leq\left\Vert f\chi
_{Q}\right\Vert _{LLogL(Q,\frac{dx}{\left\vert Q\right\vert })}\text{.}%
\end{align*}
Therefore, we have shown that%
\[
\frac{1}{\left\vert Q\right\vert }\int_{Q}\left\vert f(y)\right\vert
\log(e+\frac{\left\vert f(y)\right\vert }{\left\Vert f\chi_{Q}\right\Vert
_{L^{1}(Q,\frac{dx}{\left\vert Q\right\vert })}})dy\leq2\left\Vert f\chi
_{Q}\right\Vert _{LLogL(Q,\frac{dx}{\left\vert Q\right\vert })}.
\]
On the other hand, using successively (\ref{lux2}) and (\ref{lux1}), we
obtain
\begin{align*}
\left\Vert f\right\Vert _{LLogL(Q,\frac{dx}{\left\vert Q\right\vert })} &
=\frac{1}{\left\vert Q\right\vert }\int_{Q}\left\vert f(y)\right\vert
\log(e+\frac{\left\vert f(y)\right\vert }{\left\Vert f\right\Vert
_{LLogL(Q,\frac{dx}{\left\vert Q\right\vert })}})dy\\
&  \leq\frac{1}{\left\vert Q\right\vert }\int_{Q}\left\vert f(y)\right\vert
\log\left(  e+\frac{\left\vert f(y)\right\vert }{\left\Vert f\chi
_{Q}\right\Vert _{L^{1}(Q,\frac{dx}{\left\vert Q\right\vert })}}\right)  dy.
\end{align*}
(ii) By the definition of rearrangement,%
\begin{align}
\frac{1}{\left\vert Q\right\vert }\int_{Q}\left\vert f(y)\right\vert
\log{\small (e+}\frac{\left\vert f(y)\right\vert }{\left\Vert f\chi
_{Q}\right\Vert _{L^{1}(Q,\frac{dx}{\left\vert Q\right\vert })}}{\small )dy}
&  {\small =}\label{lux4}\\
&  \frac{1}{\left\vert Q\right\vert }\int_{0}^{\left\vert Q\right\vert
}{\small (f\chi}_{Q}{\small )}^{\ast}{\small (s)}\log{\small (e+}\frac
{(f\chi_{Q})^{\ast}(s)}{\left\Vert f\chi_{Q}\right\Vert _{L^{1}(Q,\frac
{dx}{\left\vert Q\right\vert })}}{\small )dy.}\nonumber
\end{align}
Now, since $(f\chi_{Q})^{\ast}(u)$ is decreasing, we have that, for all
$0<s<\left\vert Q\right\vert ,$
\begin{align*}
(f\chi_{Q})^{\ast}(s) &  \leq\frac{1}{s}\int_{0}^{\left\vert Q\right\vert
}(f\chi_{Q})^{\ast}(u)du\\
&  =\frac{1}{s}\left\vert Q\right\vert \left\Vert f\chi_{Q}\right\Vert
_{L^{1}(Q,\frac{dx}{\left\vert Q\right\vert })}.
\end{align*}
Inserting this information in (\ref{lux4}) we see that
\[
\frac{1}{\left\vert Q\right\vert }\int_{Q}\left\vert f(y)\right\vert
\log(e+\frac{\left\vert f(y)\right\vert }{\left\Vert f\chi_{Q}\right\Vert
_{L^{1}(Q,\frac{dx}{\left\vert Q\right\vert })}})dy\leq\frac{1}{\left\vert
Q\right\vert }\int_{0}^{\left\vert Q\right\vert }(f\chi_{Q})^{\ast}%
(s)\log(e+\frac{\left\vert Q\right\vert }{s})ds.
\]
Let $\Omega=\{s\in(0,\left\vert Q\right\vert ):\left(  \frac{e\left\vert
Q\right\vert }{s}\right)  ^{1/2}\leq\frac{(f\chi_{Q})^{\ast}(s)}{\left\Vert
f\right\Vert _{LLogL(Q,\frac{dx}{\left\vert Q\right\vert })}}\}$, then we see
that, with absolute constants, we have%
\begin{align*}
\frac{1}{\left\vert Q\right\vert }\int_{0}^{\left\vert Q\right\vert }%
(f\chi_{Q})^{\ast}(s)\log(e+\frac{\left\vert Q\right\vert }{s})ds &
\approx\frac{c}{\left\vert Q\right\vert }\int_{0}^{\left\vert Q\right\vert
}(f\chi_{Q})^{\ast}(s)\log(\frac{e\left\vert Q\right\vert }{s})ds\\
&  =\frac{c}{\left\vert Q\right\vert }\int_{\Omega}(f\chi_{Q})^{\ast}%
(s)\log(\frac{e\left\vert Q\right\vert }{s})ds\\
&  +\frac{c}{\left\vert Q\right\vert }\int_{(0.\left\vert Q\right\vert
)\backslash\Omega}(f\chi_{Q})^{\ast}(s)\log(\frac{e\left\vert Q\right\vert
}{s})ds\\
&  =(I)+(II).
\end{align*}
To estimate $(I)$ we proceed as follows\footnote{Where $c$ indicates an
absolute constant whose value may change from line to line.}%
\begin{align*}
(I) &  \leq\frac{c}{\left\vert Q\right\vert }\int_{\Omega}(f\chi_{Q})^{\ast
}(s)\log(\frac{(f\chi_{Q})^{\ast}(s)}{\left\Vert f\chi_{Q}\right\Vert
_{LLogL(Q,\frac{dx}{\left\vert Q\right\vert })}})ds\\
&  \leq\frac{c}{\left\vert Q\right\vert }\int_{0}^{\left\vert Q\right\vert
}(f\chi_{Q})^{\ast}(s)\log(\frac{(f\chi_{Q})^{\ast}(s)}{\left\Vert f\chi
_{Q}\right\Vert _{LLogL(Q,\frac{dx}{\left\vert Q\right\vert })}})ds\\
&  =\frac{c}{\left\vert Q\right\vert }\int_{Q}\left\vert f(y)\right\vert
\log(\frac{\left\vert f(y)\right\vert }{\left\Vert f\chi_{Q}\right\Vert
_{LLogL(Q,\frac{dx}{\left\vert Q\right\vert })}})dy\\
&  \leq\frac{c}{\left\vert Q\right\vert }\int_{Q}\left\vert f(y)\right\vert
\log(e+\frac{\left\vert f(y)\right\vert }{\left\Vert f\chi_{Q}\right\Vert
_{LLogL(Q,\frac{dx}{\left\vert Q\right\vert })}})dy\\
&  =c\left\Vert f\chi_{Q}\right\Vert _{LLogL(Q,\frac{dx}{\left\vert
Q\right\vert })}.
\end{align*}
Likewise,%
\begin{align*}
(II) &  =\frac{c}{\left\vert Q\right\vert }\left\Vert f\chi_{Q}\right\Vert
_{LLogL(Q,\frac{dx}{\left\vert Q\right\vert })}\int_{(0.\left\vert
Q\right\vert )\backslash\Omega}\frac{(f\chi_{Q})^{\ast}(s)}{\left\Vert
f\chi_{Q}\right\Vert _{LLogL(Q,\frac{dx}{\left\vert Q\right\vert })}}%
\log(\frac{e\left\vert Q\right\vert }{s})ds\\
&  \leq\frac{c}{\left\vert Q\right\vert }\left\Vert f\chi_{Q}\right\Vert
_{LLogL(Q,\frac{dx}{\left\vert Q\right\vert })}\int_{0}^{\left\vert
Q\right\vert }\left(  \frac{e\left\vert Q\right\vert }{s}\right)  ^{1/2}%
\log(\frac{e\left\vert Q\right\vert }{s})ds\\
&  =c\left\Vert f\chi_{Q}\right\Vert _{LLogL(Q,\frac{dx}{\left\vert
Q\right\vert })}\int_{0}^{1}\left(  \frac{e}{u}\right)  ^{1/2}\log(\frac{e}%
{u})du\\
&  =\tilde{c}\left\Vert f\chi_{Q}\right\Vert _{LLogL(Q,\frac{dx}{\left\vert
Q\right\vert })}.
\end{align*}

\end{proof}

Now we can state the version of Theorem \ref{teo:p} that corresponds to the
case $p=1$.

\begin{theorem}
\label{teo:limite}$w\in RH_{L(LogL)}$ if and only for all cubes $Q,w\chi_{Q},$
the restriction of $w$ to the cube $Q,$ belongs to $RH_{0,1}(L^{1}%
(Q),L^{\infty}(Q))$ and%
\[
\sup_{Q}\left\Vert w\chi_{Q}\right\Vert _{RH_{0,1}(L^{1}(Q),L^{\infty}%
(Q))}\approx\left\Vert w\right\Vert _{RH_{LLogL}}.
\]

\end{theorem}

\begin{proof}
For a fixed cube $Q_{0},$ we let
\[
M_{L(LogL),Q_{0}}(w\chi_{Q_{0}})(x)=\sup_{x\varepsilon Q\subset Q_{0}%
}\left\Vert w\chi_{Q_{0}}\right\Vert _{LLogL(\frac{dx}{\left\vert Q\right\vert
})(Q)}%
\]
Suppose that $w\in RH_{LLogL},$ then for $x\in Q_{0},$%
\[
M_{L(LogL),Q_{0}}(w\chi_{Q_{0}})(x)\leq\left\Vert w\right\Vert _{RH_{LLogL}%
}M_{Q_{0}}(w\chi_{Q_{0}})(x)
\]
Combining the previous estimate with the localized version of Perez's estimate
for the iterated maximal operator (cf. \cite[(13) page 174]{Perez})%
\[
M_{Q_{0}}(M_{Q_{0}}(w\chi_{Q_{0}}))(x)\leq CM_{L(LogL),Q_{0}}(w\chi_{Q_{0}%
})(x)
\]
yields%
\begin{equation}
M_{Q_{0}}(M_{Q_{0}}(w\chi_{Q_{0}}))(x)\leq C\left\Vert w\right\Vert
_{RH_{LLogL}}M_{Q_{0}}(w\chi_{Q_{0}})(x),a.e.\text{ on }Q_{0}.
\label{Mw_en_A1}%
\end{equation}
Taking rearrangements and using Herz's estimate for the maximal function we
see that, for $0<t<\left\vert Q_{0}\right\vert ,$ we have%
\begin{align*}
\frac{1}{t}\int_{0}^{t}(w\chi_{Q_{0}})^{\ast\ast}(s)ds  &  \approx\frac{1}%
{t}\int_{0}^{t}(M_{Q_{0}}(w\chi_{Q_{0}}))^{\ast}(s)ds\\
&  \approx(M_{Q_{0}}(M_{Q_{0}}(w\chi_{Q_{0}})))^{\ast}(t)\\
&  \leq C\left\Vert w\right\Vert _{RH_{LLogL}}(M_{Q_{0}}(w\chi_{Q_{0}}%
))^{\ast}(t)\\
&  \approx C\left\Vert w\right\Vert _{RH_{LLogL}}\frac{1}{t}\int_{0}^{t}%
(w\chi_{Q_{0}})^{\ast}(s)ds.
\end{align*}
In terms of $K-$functionals we therefore have that, for $0<t<\left\vert
Q_{0}\right\vert ,$%
\begin{equation}
\int_{0}^{t}K(s,w\chi_{Q_{0}};L^{1}(Q_{0}),L^{\infty}(Q_{0}))\frac{ds}{s}\leq
C\left\Vert w\right\Vert _{RH_{LLogL}}K(t,w\chi_{Q_{0}};L^{1}(Q_{0}%
),L^{\infty}(Q_{0})), \label{denuevo}%
\end{equation}
where $C$ is a universal constant. It follows from (\ref{Kcero})), that for
all cubes $Q,w\chi_{Q}\in RH_{0,1}(L^{1}(Q),L^{\infty}(Q)),$ and, moreover,
\[
\sup_{Q}\left\Vert w\chi_{Q}\right\Vert _{RH_{0,1}(L^{1}(Q),L^{\infty}%
(Q))}\leq C\left\Vert w\right\Vert _{RH_{LLogL}}.
\]
Conversely, suppose that for all cubes $Q,$ $w\chi_{Q}\in RH_{0,1}%
(L^{1}(Q),L^{\infty}(Q)),$ with $\sup_{Q}\left\Vert w\chi_{Q}\right\Vert
_{RH_{0,1}(L^{1}(Q),L^{\infty}(Q))}<\infty.$ Therefore, by (\ref{Kcero}), for
any cube $Q_{0}$, it holds%
\begin{align*}
\int_{0}^{t}K(s,w\chi_{Q_{0}};L^{1}(Q_{0}),L^{\infty}(Q_{0}))\frac{ds}{s}  &
\leq\\
&  C\left(  \sup_{Q}\left\Vert w\chi_{Q}\right\Vert _{RH_{0,1}(L^{1}%
(Q),L^{\infty}(Q))}\right)  K(t,w\chi_{Q_{0}};L^{1}(Q_{0}),L^{\infty}(Q_{0})).
\end{align*}
Let $t=\left\vert Q_{0}\right\vert .$ Then, using Example \ref{ejemplomarkao},
we obtain%
\begin{align}
\frac{1}{\left\vert Q_{0}\right\vert }\int_{0}^{\left\vert Q\right\vert
}(w\chi_{Q_{0}})^{\ast}(s)\log\frac{\left\vert Q_{0}\right\vert }{s}ds  &
\leq\nonumber\\
&  C\left(  \sup_{Q}\left\Vert w\chi_{Q}\right\Vert _{RH_{0,1}(L^{1}%
(Q),L^{\infty}(Q))}\right)  \frac{\left\Vert w\chi_{Q_{0}}\right\Vert
_{L^{1}(Q_{0})}}{\left\vert Q_{0}\right\vert }. \label{encontrada}%
\end{align}
Combining with Lemma \ref{lema:previo} it follows, that for all cubes $Q_{0},$%
\[
\left\Vert w\chi_{Q_{0}}\right\Vert _{LLogL(\frac{dx}{\left\vert
Q_{0}\right\vert },Q_{0})}\leq C\left(  \sup_{Q}\left\Vert w\chi
_{Q}\right\Vert _{RH_{0,1}(L^{1}(Q),L^{\infty}(Q))}\right)  \frac{\left\Vert
w\chi_{Q_{0}}\right\Vert _{L^{1}(Q_{0})}}{\left\vert Q_{0}\right\vert }.
\]
Consequently, $w\in RH_{LLogL},$ and
\[
\left\Vert w\right\Vert _{RH_{LLogL}}\leq C\left(  \sup_{Q}\left\Vert
w\chi_{Q}\right\Vert _{RH_{0,1}(L^{1}(Q),L^{\infty}(Q))}\right)  ,
\]
as we wished to show.
\end{proof}

\section{Reverse H\"{o}lder classes and Indices\label{RH_and_IND copy(1)}}

Interpolation theory reduces some of the basic issues around reverse
H\"{o}lder inequalities to the control of simple integrals. Although the
results in this section can be easily extended to a more general context we
will focus our development on the specific needs of this paper. So we shall
consider families of functions indexed by cubes that are constructed as
follows. For each cube $Q$ we associate a function $\phi_{Q}$ defined on
$(0,\left\vert Q\right\vert ).$ We assume that the functions $\phi_{Q}$ are
continuous, increasing and such that $\phi_{Q}(s)/s$ decreases. Let $\beta
\in\lbrack0,1),$ $q\geq1,$ we let $\phi_{Q,\beta}(s)=$ $s^{-\beta}\phi
_{Q}(s),$ and $\phi_{Q,\beta,q}(s)=[s^{-\beta}\phi_{Q}(s)]^{q};$ in
particular, $\phi_{Q,0,1}(s)=\phi_{Q}(s),$ and $\phi_{Q,\beta,1}%
(s)=\phi_{Q,\beta}(s).$ The prototype examples are constructed using the
functions $\phi_{Q}(s)$ of the form $\phi_{w,Q}(s)=K(s,w\chi_{Q}%
,L^{1}(Q),L^{\infty}(Q)),$ and their multiparameter versions $\phi
_{w,Q,\beta,q}(s)=[s^{-\beta}\phi_{w,Q}(s)]^{q},$ where $w$ is a given
weight$.$

The elementary techniques we use to estimate the integrals involving such
functions are displayed in the next Lemma. We note parenthetically (cf. part
(iii) of Lemma \ref{benson}) that the properties of the functions allow us to
achieve \textquotedblleft global control\textquotedblright\ from
\textquotedblleft local control".

Our development in this section builds extensively on the work of Samko and
her collaborators (cf. \cite{Sa}, \cite{kara}, \cite{Sa2}, and the references
therein) although the specific results dealing with families of functions are
apparently new.

We start with a definition:

\begin{definition}
A non-negative function $\phi$ on an interval $(0,l)\subset\mathbb{R}$ is said
to be almost increasing (a.i.) if there is a constant $C\geq1$ (the constant
of almost increase) such that $\phi\left(  s\right)  \leq C\phi\left(
t\right)  $ for all $s\leq t$ with $s,t\in(0,l)$.
\end{definition}

Now, we will present a lemma that will play a crucial r\^{o}le in what follows.

\begin{lemma}
\label{benson}Let $w$ be a weight and let $q\geq1,\beta\in\lbrack0,1).$The
following conditions are equivalent:

$(i)$ There exists a constant $C>0$ such that for all cubes $Q,$%
\begin{equation}
\int_{0}^{t}\phi_{w,Q,\beta,q}(s)\frac{ds}{s}\leq C\phi_{w,Q,\beta
,q}(t),\text{ for all }t\in(0,\left\vert Q\right\vert ). \label{index1}%
\end{equation}
$(ii)$ There exists $\delta>0$ such that for all cubes $Q,$ $\phi
_{w,Q,\beta,q}(s)s^{-\delta}$ is a.i. on $(0,\left\vert Q\right\vert ),$ with
constant of almost increase independent of $Q$ and $\beta.$

$(iii)$ There exists $\delta>0,\gamma\in(0,1)$ such that for all cubes $Q,$
$\phi_{w,Q,\beta,q}(s)s^{-\delta}$ is a.i. on $(0,\gamma\left\vert
Q\right\vert ),$ with constant of almost increase independent of $Q$ and
$\beta.$
\end{lemma}

\begin{proof}
$(i)\Longrightarrow(ii).$ This is an elementary differential inequality
argument (e.g. cf. \cite{M}) which we include for the sake of completeness.
Let%
\[
F_{w,Q,\beta,q}(t)=\int_{0}^{t}\phi_{w,Q,\beta,q}(s)\frac{ds}{s}.
\]
Then $(i)$ can be rewritten as%
\[
F_{w,Q,\beta,q}(t)\leq Ct(F_{w,Q,\beta,q}(t))^{^{\prime}}.
\]
Therefore,%
\[
\left(  \frac{1}{C}\ln t\right)  ^{\prime}\leq\left(  \ln F_{w,Q,\beta
,q}(t)\right)  ^{\prime},
\]
so that for $0<x<y<\left\vert Q\right\vert ,$ we have%
\[
\ln\left(  \frac{y}{x}\right)  ^{1/C}\leq\ln\frac{F_{w,Q,\beta,q}%
(y)}{F_{w,Q,\beta,q}(x)},
\]
yielding,%
\[
x^{-1/C}F_{w,Q,\beta,q}(x)\leq y^{-1/C}F_{w,Q,\beta,q}(y)\leq y^{-1/C}%
C\phi_{w,Q,\beta,q}(y).
\]
Combining the last inequality with,%
\begin{align*}
F_{w,Q,\beta,q}(x)  &  =\int_{0}^{x}s^{q(1-\beta)}[\frac{\phi_{w,Q}(s)}%
{s}]^{q}\frac{ds}{s}\\
&  \geq\left(  \frac{\phi_{w,Q}(x)}{x}\right)  ^{q}\frac{x^{q(1-\beta)}%
}{q(1-\beta)}\\
&  =\frac{\phi_{w,Q,\beta,q}(x)}{q(1-\beta)},
\end{align*}
implies%
\begin{align*}
x^{-1/C}\phi_{w,Q,\beta,q}(x)  &  \leq(1-\beta)qCy^{-1/C}\phi_{w,Q,\beta
,q}(y)\\
&  \leq qCy^{-1/C}\phi_{Q,\beta,q}(y).
\end{align*}

$(ii)\Rightarrow(iii).$ Is trivial.

$(iii)\Rightarrow(i).$ Suppose that there exists $\delta>0,\gamma\in(0,1)$
such that for all $Q$, $\phi_{w,Q,\beta,q}(s)s^{-\delta}$ is almost increasing
on $(0,\gamma\left\vert Q\right\vert )$, with constant of a.i. $C,$
independent of $Q$ and $\beta.$ Consider two cases. If $t<\gamma\left\vert
Q\right\vert ,$ then we can write
\begin{align*}
\int_{0}^{t}\phi_{w,Q,\beta,q}(s)\frac{ds}{s}  &  =\int_{0}^{t}\phi
_{w,Q,\beta,q}(s)s^{-\delta}s^{\delta}\frac{ds}{s}\\
&  \leq C\phi_{w,Q,\beta,q}(t)t^{-\delta}\frac{t^{\delta}}{\delta}\\
&  =\frac{C}{\delta}\phi_{w,Q,\beta,q}(t).
\end{align*}

Now suppose that $t\in(\gamma\left\vert Q\right\vert ,\left\vert Q\right\vert
).$ Then%
\begin{align*}
\int_{0}^{t}\phi_{w,Q,\beta,q}(s)\frac{ds}{s}  &  =\int_{0}^{\gamma\left\vert
Q\right\vert }\phi_{w,Q,\beta,q}(s)\frac{ds}{s}+\int_{\gamma\left\vert
Q\right\vert }^{t}\phi_{w,Q,\beta,q}(s)\frac{ds}{s}\\
&  =(I)+(II).
\end{align*}
By the first part of the proof
\begin{align*}
(I)  &  =\int_{0}^{\gamma\left\vert Q\right\vert }\phi_{w,Q,\beta,q}%
(s)\frac{ds}{s}\\
&  \leq\frac{C}{\delta}\phi_{w,Q,\beta,q}(\gamma\left\vert Q\right\vert )\\
&  =\frac{C}{\delta}\left(  \phi_{w,Q}(\gamma\left\vert Q\right\vert )\right)
^{q}\gamma^{-q\beta}\left\vert Q\right\vert ^{-q\beta}\\
&  \leq\frac{C}{\delta}\left(  \phi_{w,Q}(t)\right)  ^{q}\gamma^{-\beta
q}t^{-q\beta}\text{ \ \ \ \ (since }\phi_{w,Q}\text{ increases and
}t<\left\vert Q\right\vert \text{)}\\
&  =\frac{C}{\delta}\gamma^{-\beta q}\phi_{w,Q,\beta,q}(t).
\end{align*}
To estimate the remaining integral we use successively that $\phi_{w,Q}$
increases and $t<\left\vert Q\right\vert ,$ to obtain
\begin{align*}
(II)  &  =\int_{\gamma\left\vert Q\right\vert }^{t}\phi_{w,Q,\beta,q}%
(s)\frac{ds}{s}\\
&  =\int_{\gamma\left\vert Q\right\vert }^{t}\left(  \phi_{w,Q}(s)\right)
^{q}s^{-\beta q}\frac{ds}{s}\\
&  \leq\left(  \phi_{w,Q}(t)\right)  ^{q}\int_{\gamma\left\vert Q\right\vert
}^{t}s^{-\beta q-1}ds\\
&  \leq\left(  \phi_{w,Q}(t)\right)  ^{q}\frac{1}{\beta q}\frac{t^{\beta
q}-(\gamma\left\vert Q\right\vert )^{\beta q}}{t^{\beta q}(\gamma\left\vert
Q\right\vert )^{\beta q}}\\
&  \leq\phi_{w,Q,\beta,q}(t)\frac{1}{\beta q}\frac{1-\gamma^{\beta q}}%
{\gamma^{\beta q}}.
\end{align*}
Combining the estimates for $(I)$ and $(II)$ we obtain
\[
\int_{0}^{t}\phi_{w,Q,\beta}(s)\frac{ds}{s}\leq\left(  \frac{C}{\delta}%
\gamma^{-\beta q}+\frac{1}{\beta q}\frac{1-\gamma^{\beta q}}{\gamma^{\beta q}%
}\right)  \phi_{w,Q,\beta}(t).
\]
But it is easy to obtain a bound independent of $\beta$ on the right hand
side. Indeed, by elementary calculus we see that the function $f\left(
x\right)  =x\ln\left(  x\right)  -\frac{1}{\beta q}\left(  x^{q\beta
}-1\right)  $ is increasing on $[1,+\infty)$ and $f(1)=0,$ therefore
$\gamma^{-1}\ln\left(  \gamma^{-1}\right)  \geq$ $f(\frac{1}{\gamma})>\frac
{1}{\beta q}\left(  \frac{1-\gamma^{\beta q}}{\gamma^{\beta q}}\right)  ,$
while $\frac{C}{\delta}\gamma^{-\beta q}\leq\frac{C}{\delta}\gamma^{-1}.$
Therefore, we obtain
\[
\int_{0}^{t}\phi_{w,Q,\beta,q}(s)\frac{ds}{s}\leq\left(  \frac{C}{\delta
}\gamma^{-1}+\gamma^{-1}\ln\left(  \gamma^{-1}\right)  \right)  \phi
_{w,Q,\beta,q}(t),
\]
and the desired result follows.
\end{proof}

The preceding Lemma combined with the work of Samko and her collaborators (cf.
(\ref{jardin}) below) motivated the following definition

\begin{definition}
\label{def:estrela}Let $w$ be a given weight and let $\beta\in\lbrack
0,1),q\geq1.$ Consider family of functions $\{\phi_{w,Q,\beta,q}\}_{Q}$ as
above. We define the index $ind$ of $\{\phi_{w,Q,\beta,q}\}_{Q}$ as follows,%
\begin{align}
ind\{\phi_{w,Q,\beta,q}\}_{Q}  &  =\sup\{\delta\geq0:\exists\gamma
\in(0,1)\text{ such that \ for all cubes }Q\text{, }\phi_{w,Q,\beta
,q}(s)s^{-\delta}\text{ is a.i.}\label{laindicada}\\
&  \text{ on }(0,\gamma\left\vert Q\right\vert )\text{, with constant of a.i.
independent of }Q\}\nonumber
\end{align}

When $\beta=0,q=1,$ we put $\phi_{w,Q,0,1}(s):=\phi_{w,Q}(s);$ then note that
$ind\{\phi_{w,Q}\}_{Q}=ind\{\phi_{w,Q,0,1}\}_{Q}.$
\end{definition}

The same definition applies when dealing with a single function $\phi
_{\beta,q}(s)=\left(  s^{-\beta}\phi(s)\right)  ^{q},$ where, for the sake of
comparison, we assume that $\phi$ is such that $\phi(s)$ increases and
$\phi(s)/s$ decreases on $(0,l)$. For single functions we use the following
compatible definition (cf. \cite[Theorem 3.6, pag 448]{kara})%
\begin{equation}
i\{\phi_{\beta,q}\}=\sup\{\delta\geq0:\text{ }\phi_{\beta,q}(s)s^{-\delta
}\text{ is a.i. on }(0,l)\}. \label{jardin}%
\end{equation}
The following remark will be useful in what follows

\begin{remark}
\label{remark:remarkao}Let $w$ be a weight and let $q\geq1,\beta\in
\lbrack0,1).$ Then,%
\[
ind\{\phi_{w,Q,\beta}\}_{Q}>0\Leftrightarrow ind\{\phi_{w,Q,\beta,q}\}_{Q}>0.
\]
Likewise,%
\[
i\{\phi_{\beta,q}\}>0\Leftrightarrow i\{\phi_{\beta}\}>0.
\]

\end{remark}

\begin{proof}
The result follows directly from Definitions \ref{def:estrela} and
(\ref{jardin}). For example, note that if $\phi_{Q,\beta}(s)s^{-\delta}$ is
a.i. then $(\phi_{Q,\beta}(s))^{q}s^{-\tilde{\delta}}$ is a.i., with
$\tilde{\delta}=\delta q;$ and conversely if $(\phi_{Q,\beta}(s))^{q}%
s^{-\delta}$ is a.i, then $\phi_{Q,\beta}(s)s^{-\delta/q}$ is a.i..
\end{proof}

With this definition we can now reformulate Lemma \ref{benson} as follows

\begin{proposition}
\label{aero}Let $w$ be a weight, and let $q\geq1,\beta\in\lbrack0,1).$ The
following are equivalent:

(i) There exists $C>0$ independent of $Q$ and $\beta$ such that for all
\[
\int_{0}^{t}\phi_{w,Q,\beta,q}(s)\frac{ds}{s}\leq C\phi_{w,Q,\beta
,q}(t),\text{ for all }t\in(0,\left\vert Q\right\vert ).
\]
(ii)
\begin{equation}
ind\{\phi_{w,Q,\beta,q}\}_{Q}>0. \label{index2}%
\end{equation}

\end{proposition}

\begin{proof}
Suppose (i) holds. Then, by Lemma \ref{benson} (iii), there exists $\delta>0$
and $\gamma\in(0,1)$ such that for all $Q,$ $\phi_{w,Q,\beta,q}(s)s^{-\delta}$
is a.i. on $(0,\gamma\left\vert Q\right\vert ),$ with constant of a.i.
independent of $Q.$ Therefore, (\ref{index2}) holds directly from Definition
\ref{def:estrela}. Likewise, if (\ref{index2}) holds then Lemma \ref{benson}
(iii) holds, and therefore (i) holds.
\end{proof}

We now show that, in some sense, the computation of $ind\{\phi_{w,Q,\beta
}\}_{Q}$ can be reduced to the computation of $ind\{\phi_{w,Q}\}_{Q}$

\begin{proposition}
\label{switch}%
\[
ind\{\phi_{w,Q,\beta,q}\}_{Q}>0\Leftrightarrow ind\{\phi_{w,Q,\beta}%
\}_{Q}>0\Leftrightarrow ind\{\phi_{w,Q}\}_{Q}>\beta.
\]

\end{proposition}

\begin{proof}
The first equivalence was proved in Remark \ref{remark:remarkao}. We therefore
only need to prove the second equivalence. Towards this end let us fix an
arbitrary cube $Q.$ The case $\beta=0$ holds by definition since $\phi
_{w,Q,0}=\phi_{w,Q}.$ Therefore we shall now assume that $\beta>0.$ Suppose,
moreover, that $ind\{\phi_{w,Q}\}_{Q}>\beta,$ then we can find $\delta
>0,\gamma\in(0,1),$ such that $\delta>\beta$ and $\phi_{w,Q}(s)s^{-\delta}$ is
a.i. on $(0,\gamma\left\vert Q\right\vert ).$ Therefore, since%
\[
\phi_{w,Q,\beta}(s)s^{-(\delta-\beta)}=\phi_{w,Q}(s)s^{-\delta}%
\]
is almost increasing on $(0,\gamma\left\vert Q\right\vert ),$ with
$\delta-\beta>0,$ and since $Q$ was arbitrary, we see that $ind\{\phi
_{w,Q,\beta}\}_{Q}>0.$ Conversely, if $ind\{\phi_{w,Q,\beta}\}_{Q}>0,$ then we
can find $\delta>0$ such that for any cube $Q,$ $\phi_{w,Q,\beta}%
(s)s^{-\delta}=\phi_{w,Q}(s)s^{-(\delta+\beta)}$ is a.i. on $(0,\gamma
\left\vert Q\right\vert )$ for some fixed $\gamma\in(0,1).$ Therefore, since%
\[
\phi_{w,Q}(s)s^{-\beta}=\left(  \phi_{w,Q}(s)s^{-(\delta+\beta)}\right)
s^{\delta}%
\]
we see that%
\[
ind\{\phi_{w,Q}\}_{Q}>\beta.
\]

\end{proof}

The usual definitions of indices in the literature concern the index of one
function. We shall now compare the results in this section with classical
results using the more common definitions of indices. For comparison
purposes\footnote{The results for the classical indices are valid under less
restrictive conditions.} we let $\phi$ be defined on $(0,l),$ with $\phi$
increasing and $\phi(s)/s$ decreasing. Then many definitions are equivalent.
Here we shall specialize our results and consider only functions of the form
$\Psi(s)=(s^{-\beta}\phi(s))^{q}.$ Let (cf. \cite{BaSt}, \cite{Sa}),%
\[
\alpha_{\Psi}=\sup_{x>1}\frac{\ln\left(  \underrightarrow{\lim}%
_{h\longrightarrow0}\frac{\Psi(xh)}{\Psi(h)}\right)  }{\ln x}.
\]
Then we have the classical result (cf. \cite{BaSt}, \cite{Krein}, \cite{Mal},
\cite{MM}, \cite{Sa}, and the references therein giving the same result under
different definitions of indices)

\begin{lemma}
\label{benson2}The following are equivalent:

(i) There exists a constant $C>0$ such that
\begin{equation}
\int_{0}^{t}\Psi(s)\frac{ds}{s}\leq C\Psi(t),\text{ for all \ }t\in(0,l).
\label{aero2}%
\end{equation}

(ii) $\alpha_{\Psi}>0.$
\end{lemma}

Combining Proposition \ref{aero} and Lemma \ref{benson2} we see that our
definition of index of a single function (\ref{jardin}) is compatible with the
classical ones.

\begin{corollary}
\label{aero5}Let $\Psi$ be a function defined on $(0,l)$ as above. Then,%
\[
\alpha_{\Psi}>0\Leftrightarrow i\{\phi_{\Psi}\}>0.
\]

\end{corollary}

\begin{proof}
The proof of Proposition (\ref{aero}) for single functions gives us that
(\ref{aero2}) holds if and only if $i\{\phi_{\Psi}\}>0.$ On the other hand, by
Lemma \ref{benson2} we know that (\ref{aero2}) holds if and only if
$\alpha_{\Psi}>0.$ The result follows.
\end{proof}

\begin{example}
The compatibility of the index (for a single function) with the classical
indices is discussed in \cite{kara}. In this example we show a simple
calculation that hints the reason why the index considered here coincides with
classical indices for the classes of functions under consideration. Suppose
that $\phi(s)s^{-\delta}$ is almost increasing (a.i.), then, for some constant
$c\geq1,$ we have that for $x>1,$%
\[
(xh)^{-\delta}\phi(xh)\geq\frac{1}{c}h^{-\delta}\phi(h).
\]
It follows that%
\begin{align*}
\alpha_{\phi}  &  =\sup_{x>1}\frac{\ln\left(  \underrightarrow{\lim
}_{h\longrightarrow0}\frac{\phi(xh)}{\phi(h)}\right)  }{\ln x}\\
&  \geq\sup_{x>1}\frac{\ln\left(  \frac{1}{c}x^{\delta}\right)  }{\ln x}\\
&  =\sup_{x>1}\{\delta+\frac{\ln(\frac{1}{c})}{\ln x}\}\\
&  =\delta.
\end{align*}

\end{example}

\subsection{Characterization of abstract reverse H\"{o}lder classes via
indices}

In this section we essentially show how the results of Mastylo-Milman
\cite{MM} can be obtained using the indices we have introduced in this paper.

\begin{theorem}
\label{teo:ind}Let $\vec{X}$ be a Banach pair, $\theta\in(0,1),$ and $q\geq1.$
Then,%
\[
RH_{\theta,q}(\vec{X})=\{w\in X_{0}+X_{1}:i\{K(\cdot,w;\vec{X})\}>\theta\}.
\]

\end{theorem}

\begin{proof}
We shall use the following special case of Holmstedt's formula (cf.
\cite[Corollary 3.6.2 (b), pag 53]{BL}) (with constants dependent on
$\theta,q$ but not on $w$)%
\begin{equation}
K(t,w;\vec{X}_{\theta,q},X_{1})\approx\left\{  \int_{0}^{t^{1/(1-\theta)}%
}[s^{-\theta}K(s,w;\vec{X})]^{q}\frac{ds}{s}\right\}  ^{1/q}. \label{holm}%
\end{equation}
Fix $(\theta,q)\in(0,1)\times\lbrack1,\infty).$ Let $w\in RH_{\theta,q}%
(\vec{X}).$ Using (\ref{holm}) we can rewrite (\ref{Ktheta}) as%
\begin{equation}
\left\{  \int_{0}^{t^{1/(1-\theta)}}[s^{-\theta}K(s,w;\vec{X})]^{q}\frac
{ds}{s}\right\}  ^{1/q}\leq C\left\Vert w\right\Vert _{RH_{\theta,q}(\vec{X}%
)}t\frac{K(t^{\frac{1}{1-\theta}},w,\vec{X})}{t^{\frac{1}{1-\theta}}},\forall
t>0, \label{holm2}%
\end{equation}
which simplifies to%
\begin{equation}
\int_{0}^{t}[s^{-\theta}K(s,w;\vec{X})]^{q}\frac{ds}{s}\leq C\left\Vert
w\right\Vert _{RH_{\theta,q}(\vec{X})}^{q}[t^{-\theta}K(t,w,\vec{X}%
)]^{q},\forall t>0. \label{holm1}%
\end{equation}
Since $K(s,w;\vec{X})$ increases and $\frac{K(s,w;\vec{X})}{s}$ decreases, by
Lemma \ref{benson2} and Corollary \ref{aero5} we have
\begin{equation}
i\{[(\cdot)^{-\theta}K(\cdot,w;\vec{X})]^{q}\}>0. \label{holm3}%
\end{equation}
Consequently, by Proposition \ref{switch}, it follows that
\[
i\{K(\cdot,w;\vec{X})\}>\theta.
\]
It is easy to see that all the steps can be now reversed. Indeed, if the
previous inequality holds, then, by Proposition \ref{switch}, we see that
(\ref{holm3}) holds and, by Lemma \ref{benson2}, we find that (\ref{holm1})
holds for all $t>0.$ Changing $t\rightarrow t^{1/(1-\theta)}$ in the resulting
inequality, we successively see that (\ref{holm2}), (\ref{holm}), and,
finally, (\ref{Ktheta}) hold, as we wished to show.
\end{proof}

\begin{remark}
\label{introdespues} Note that the second index does not appear in the
abstract characterization of $RH_{\theta,q}(\vec{X}),$ therefore it follows
that, for all $q\geq1,$%
\begin{align*}
RH_{\theta,q}(\vec{X})  &  =\{w\in X_{0}+X_{1}:i\{K(\cdot,w;\vec{X}%
)\}>\theta\}\\
&  =RH_{\theta,1}(\vec{X}).
\end{align*}

\end{remark}

The previous analysis also yields the following characterization of the
limiting class $RH(\vec{X})$ defined by
\[
RH(\vec{X})=\bigcup\limits_{(\theta,q)\in(0,1)\times\lbrack1,\infty
)}RH_{\theta,q}(\vec{X}).
\]

\begin{theorem}
\label{teo:teo}%
\[
RH(\vec{X})=\{w:i\{K(\cdot,w;\vec{X})\}>0\}.
\]

\end{theorem}

\begin{proof}
Let $w\in\bigcup\limits_{(\theta,q)\in(0,1)\times\lbrack1,\infty)}%
RH_{\theta,q}(\vec{X}).$ It follows that there exist $\theta\in(0,1),$ and
$q\geq1,$ such that $x\in RH_{\theta,q}(\vec{X}).$ Therefore, by the previous
theorem, $i\{K(\cdot,w;\vec{X})\}>\theta>0$.

Conversely, suppose that%
\[
i\{K(\cdot,w;\vec{X})\}>0.
\]
Let $q\geq1,$ and select $\theta$ such that $i\{K(\cdot,w;\vec{X}%
)\}>\theta>0.$ Then, by Theorem \ref{teo:ind}, $w\in RH_{\theta,q}(\vec
{X})\subset RH(\vec{X}).$
\end{proof}

The limiting case $\theta=0$ can be obtained using the same arguments.

\begin{corollary}
\label{coro:marcao}Let $\vec{X}$ be an ordered Banach pair. Then%
\[
RH_{0,1}(\vec{X})=RH(\vec{X}).
\]

\end{corollary}

\begin{proof}
Let $n$ be the norm of the embedding, $X_{1}\subset X_{0}.$ By definition,
$w\in RH_{0,1}(\vec{X})$ if and only for all $0<t<n$,
\begin{equation}
\int_{0}^{t}K(s,w;\vec{X})\frac{ds}{s}\leq cK(t,w;\vec{X}). \label{ordinaria}%
\end{equation}
By Lemma \ref{benson2} and Corollary \ref{aero5}, (\ref{ordinaria}) is
equivalent to%
\[
i\{K(\cdot,w;\vec{X})\}>0.
\]
Consequently, by Theorem \ref{teo:teo}, $w\in RH(\vec{X})$.

Conversely, if $w\in RH(\vec{X}),$ then, by Theorem \ref{teo:teo},
$i\{K(\cdot,w;\vec{X})\}>0,$ consequently, by Lemma \ref{benson2}, we find
that (\ref{ordinaria}) holds, whence $w\in RH_{0,1}(\vec{X}).$
\end{proof}

In this framework Gehring's Lemma is a triviality

\begin{theorem}
(Gehring's Lemma) (i) Let $\theta\in(0,1),1\leq q<\infty.$ Suppose that $w\in
RH_{\theta,q}(\vec{X}),$ then there exists $\theta^{\prime}>\theta,$ such
that, for all $1<p<\infty,$ $w\in RH_{\theta^{\prime},p}(\vec{X}).$

(ii) Suppose that $w\in RH_{0,1}(\vec{X})$ then there exists $\theta^{\prime
}>0,$ $1<p<\infty,$ such that $w\in RH_{\theta^{\prime},p}(\vec{X}).$
\end{theorem}

\begin{proof}
(i) By Theorem \ref{teo:ind}, $i\{K(\cdot,w;\vec{X})\}>\theta.$ Pick
$\theta^{\prime}\in(\theta,i\{K(\cdot,w;\vec{X})\})$ then by Theorem
\ref{teo:ind}, $w\in RH_{\theta^{\prime},p}(\vec{X}),$ for all $p>1.$

(ii) Follows directly from Corollary \ref{coro:marcao}.
\end{proof}

\subsection{Classical Reverse H\"{o}lder classes and indices}

In this section we characterize the classical classes of weights that satisfy
reverse H\"{o}lder inequalities. The results are completely analogous to the
ones in the previous section but the characterizations are now given in terms
of indices of families of $K$-functionals of the weights involved.

Let $w$ be a weight and let $p\in\lbrack1,\infty).$ The family of functions we
use here can be defined as follows. Let $w$ be a weight and for each cube $Q$,
let $\phi_{w,Q}(s)=K(s,w\chi_{Q},L^{1},L^{\infty});$ and moreover, let%
\begin{align*}
\phi_{w,Q,1/p^{\prime}}(s)  &  =s^{-1/p^{\prime}}K(s,w\chi_{Q},L^{1}%
,L^{\infty}),0<s<\left\vert Q\right\vert .\\
\phi_{w,Q,1/p^{\prime},q}(s)  &  =\left(  s^{-1/p^{\prime}}K(s,w\chi_{Q}%
,L^{1},L^{\infty})\right)  ^{q},0<s<\left\vert Q\right\vert .
\end{align*}

Combining the above results with the characterization of classical reverse
H\"{o}lder inequalities in terms of abstract H\"{o}lder classes that were
given in Section \ref{sec:classicalvs}, we obtain

\begin{theorem}
\label{teo:galvanizado}(i) Let $p>1,$ then
\[
RH_{p}=\{w:ind\{\phi_{w,Q,1/p^{\prime},p}\}_{Q}>0\}=\{w:ind\{\phi_{w,Q}%
\}_{Q}>1/p^{\prime}\}
\]

(ii)
\[
RH=\{w:ind\{\phi_{w,Q}\}_{Q}>0\}
\]

(iii)%
\[
RH=RH_{LLogL}.
\]

\end{theorem}

\begin{proof}
\textbf{(i) }Suppose that $w\in RH_{p}.$ Then, by Theorem \ref{teo:p} we have
that for all cubes $Q,$ $w\chi_{Q}\in RH_{1/p^{\prime},p}(L^{1}(Q),L^{\infty
}(Q))$ and
\[
\sup_{Q}\left\Vert w\chi_{Q}\right\Vert _{RH_{1/p^{\prime},p}(L^{1}%
(Q),L^{\infty}(Q))}\approx\left\Vert w\right\Vert _{RH_{p}}.
\]
By Theorem \ref{teo:p} there exists a constant $c>0,$ such that for all cubes
$Q,$%
\[
\int_{0}^{t}\phi_{Q,1/p^{\prime},p}(s)\frac{ds}{s}\leq c\left\Vert
w\right\Vert _{RH_{p}}^{p}\phi_{Q,1/p^{\prime},p}(s),\text{ \ }0<t<\left\vert
Q\right\vert .
\]
Consequently, by Proposition \ref{aero} followed by Proposition \ref{switch},%
\[
ind\{\phi_{w,Q}\}_{Q}>1/p^{\prime}.
\]
Conversely, suppose that $w$ is such that $ind\{\phi_{w,Q,1/p^{\prime}%
,p}\}_{Q}>0.$ Then, by Proposition \ref{aero}, there exists a constant $c>0$
such that
\[
\int_{0}^{t}\phi_{Q,1/p^{\prime},p}(s)\frac{ds}{s}\leq c\phi_{Q,1/p^{\prime
},p}(s),\text{ \ }0<t<\left\vert Q\right\vert .
\]
It follows by Theorem \ref{teo:p} that
\[
\sup\limits_{Q}\left\Vert w\chi_{Q}\right\Vert _{RH_{1-1/p,p}(L^{1}%
(Q),L^{\infty}(Q))}\preceq c,
\]
and, moreover, $w\in RH_{p},$ with%
\[
\left\Vert w\right\Vert _{RH_{p}}\preceq c.
\]

\textbf{(ii)} Suppose that $w\in RH,$ then $w\in RH_{p}$ for some $p$. Then,
by part (i), $ind\{\phi_{w,Q}\}_{Q}>1/p^{\prime}>0.$ Conversely, if \ $w$ is a
weight such that $\ ind\{\phi_{w,Q}\}_{Q}>0,$ then we can select $p>1$ close
enough to $1$ so that $\frac{1}{p^{\prime}}=1-\frac{1}{p}<ind\{\phi
_{w,Q}\}_{Q}.$ Therefore, by (i), $w\in RH_{p}\subset RH.$

\textbf{(iii)} We show first the inclusion $RH\subset RH_{LLogL}$. Suppose
that $w\in RH,$ then there exists $p>1$ such that $w\in RH_{p}.$ Now, it is
easy to verify that for  $0<\alpha<1,$ we have $\log(e+\frac{1}{x})\preceq
x^{-\alpha}$, $x\in(0,1);$ consequently, by H\"{o}lder's inequality, we have%
\begin{align*}
\frac{1}{\left\vert Q\right\vert }\int_{0}^{\left\vert Q\right\vert }%
(w\chi_{Q})^{\ast}(s)\log(e+\frac{\left\vert Q\right\vert }{s})ds &
\preceq(\frac{1}{\left\vert Q\right\vert }\int_{0}^{\left\vert Q\right\vert
}(w\chi_{Q})^{\ast}(s)^{p}ds)^{1/p}(\frac{1}{\left\vert Q\right\vert }\int%
_{0}^{\left\vert Q\right\vert }\left\vert Q\right\vert ^{-\alpha p^{\prime}%
}s^{\alpha p^{\prime}}ds)^{1/p^{\prime}}\\
&  \preceq(\frac{1}{\left\vert Q\right\vert }\int_{0}^{\left\vert Q\right\vert
}(w\chi_{Q})^{\ast}(s)^{p}ds)^{1/p}\\
&  \preceq\frac{1}{\left\vert Q\right\vert }\int_{0}^{\left\vert Q\right\vert
}(w\chi_{Q})^{\ast}(s)ds\text{ (since }w\in RH_{p}\text{).}%
\end{align*}
Therefore, \ the conclusion follows from Lemma \ref{lema:previo}.

We now prove the opposite inclusion. Suppose that $w\in RH_{LLogL}.$ Then, by
Theorem \ref{teo:limite}, for all cubes $Q$ we have that $w\chi_{Q}\in
RH_{0,1}\left(  L^{1}\left(  Q\right)  ,L^{\infty}\left(  Q\right)  \right)  $
and, moreover, $\sup_{Q}\left\Vert w\chi_{Q}\right\Vert _{RH_{0,1}%
(L^{1}(Q),L^{\infty}(Q))}\approx\left\Vert w\right\Vert _{RH_{LLogL}}$. It
follows that%
\[
\int_{0}^{t}K(s,w\chi_{Q};L^{1}\left(  Q\right)  ,L^{\infty}\left(  Q\right)
)\frac{ds}{s}\leq\left\Vert w\right\Vert _{RH_{LLogL}}K(t,w\chi_{Q}%
;L^{1}\left(  Q\right)  ,L^{\infty}\left(  Q\right)  ).
\]
Consequently, by Proposition \ref{aero}%
\[
ind\{K(\cdot,w\chi_{Q};L^{1}\left(  Q\right)  ,L^{\infty}\left(  Q\right)
)\}_{Q}>0.
\]

\end{proof}

\subsection{Non-doubling weights\label{sec:doubling}}

Let $1<p<\infty.$ For a locally integrable positive function $w$, we define
the class of reverse H\"{o}lder weights $RH_{p}(w)$ simply replacing $dx$ by
$w(x)dx$ in the definition of $RH_{p}.$ Thus, we say that $g\in RH_{p}(w),$ if
there exists $C\geq1$ such that for every cube $Q,$ with sides parallel to the
coordinate axes, we have
\[
\left(  \frac{1}{w(Q)}\int_{Q}g(x)^{p}w(x)dx\right)  ^{1/p}\leq\frac{C}%
{w(Q)}\int_{Q}g(x)w(x)dx,
\]
where $w(Q)=\int_{Q}w(x)dx$. If the measure $\mu:=$ $w(x)dx$ satisfies a
doubling condition, i.e., if there exists a constant $c>0$ such that
$\mu(B(x,2r))\leq c\mu(B(x,r)),$ then for the maximal operator $M_{w}$%
\[
M_{w}g(x)=\sup_{Q\ni x}\frac{1}{w(Q)}\int_{Q}g(x)w(x)dx
\]
we have the equivalence\footnote{here the rearrangements $f_{w}^{\ast}%
,f_{w}^{\ast\ast},$ etc, are with respect to the measure $w(x)dx$}%
\begin{equation}
(M_{w}g)_{w}^{\ast}(t)\approx g_{w}^{\ast\ast}(t), \label{nd1}%
\end{equation}
which in turn gives us%
\[
(M_{w}g)_{w}^{\ast}(t)\approx\frac{K(t,g;L_{w}^{1},L_{w}^{\infty})}{t}.
\]
Therefore, we can use the analysis of the previous sections \textit{mutatis
mutandi}. On the other hand, when dealing with non-doubling measures
(\ref{nd1}) may fail. In particular, the equivalence%
\[
tg_{w}^{\ast\ast}(t)\approx K(t,g;L_{w}^{1},L_{w}^{\infty}),
\]
may not hold, so that the implementation of the interpolation method, as
discussed in previous sections, requires a different approach. Fortunately, it
is possible to find an alternative formula for the $K-$functional that
resolves this obstacle, as explained in the Introduction (cf. (\ref{spi}),
(\ref{nn}) and (\ref{nn1}) for the relevant formulae). In particular, from the
definitions given in the Introduction, for each packing $\pi,$%

\[
S_{\pi}(g)(x)=\sum_{i=1}^{\left\vert \pi\right\vert }\left(  \frac{1}%
{w(Q_{i})}\int_{Q_{i}}g(y)w(y)dy\right)  \chi_{Q_{i}}(x),\ \ \ \ g\in
L_{w}^{1}(\mathbb{R}^{n})+L^{\infty}(\mathbb{R}^{n}),
\]
and for $1\leq p<\infty,$ we have (cf. \cite{aa})
\[
K(t,g;L_{w}^{p},L^{\infty})\approx t^{1/p}\sup\limits_{\pi}\left(  S_{\pi
}(\left\vert g\right\vert ^{p})\right)  _{w}^{\ast}(t).
\]
The import of this construction is that it allows to translate the information
provided by the definition of a reverse H\"{o}lder inequality in the language
of $K-$functionals, while avoiding the use of the possibly unbounded maximal
operator $M_{w}.$ Indeed, suppose that $g\in RH_{p}(w(x)dx),$ then, directly
from the definitions, we see that%
\[
t^{1-1/p}K(t^{1/p},g;L_{w}^{p}(\mathbb{R}^{n}),L^{\infty}(\mathbb{R}^{n}))\leq
C\left\Vert g\right\Vert _{RH_{p}(w(x)dx)}K(t,g;L_{w}^{1}(\mathbb{R}%
^{n}),L^{\infty}(\mathbb{R}^{n})).
\]
Consequently, if we let $\theta=1-1/p,$ we get%
\[
t^{\theta}K(t^{1-\theta},g;L_{w}^{p}(\mathbb{R}^{n}),L^{\infty}(\mathbb{R}%
^{n}))\leq C\left\Vert g\right\Vert _{RH_{p}(w(x)dx)}K(t,g;L_{w}%
^{1}(\mathbb{R}^{n}),L^{\infty}(\mathbb{R}^{n}))
\]
which finally gives%
\begin{align}
&  K(t^{1-\theta},g;L_{w}^{p}(\mathbb{R}^{n}),L^{\infty}(\mathbb{R}%
^{n}))\label{conpeso}\\
&  \leq C\left\Vert g\right\Vert _{RH_{p}(w(x)dx)}t^{-\frac{\theta}{1-\theta}%
}K(t^{\frac{1}{1-\theta}},g;L_{w}^{1}(\mathbb{R}^{n}),L^{\infty}%
(\mathbb{R}^{n}))\\
&  =C\left\Vert g\right\Vert _{RH_{p}(w(x)dx)}t\frac{K(t^{\frac{1}{1-\theta}%
},g;L_{w}^{1}(\mathbb{R}^{n}),L^{\infty}(\mathbb{R}^{n}))}{t^{\frac
{1}{1-\theta}}}.\nonumber
\end{align}
Therefore, (\ref{Ktheta}) holds and we have: $g\in RH_{p}(w(x)dx)\Rightarrow
g\in RH_{\theta,p}(L_{w}^{1}(\mathbb{R}^{n}),L^{\infty}(\mathbb{R}^{n})).$ We
can localize this result using the corresponding local formula for the
$K-$functional (cf. \cite{MarM}): for all cubes $Q$%
\[
K(t^{1-\theta},g\chi_{Q};L_{w}^{p}(Q),L^{\infty}(Q))\leq C\left\Vert
g\right\Vert _{RH_{p}(w(x)dx)}t\frac{K(t^{\frac{1}{1-\theta}},g\chi_{Q}%
;L_{w}^{1}(Q),L^{\infty}(Q))}{t^{\frac{1}{1-\theta}}}.
\]
Using the Holmstedt's formula we can proceed with the analysis in the weighted
case as we did in the unweighted case. In fact, we obtain that if $g\in
RH_{p}(w(x)dx)$ then, for $0<t<w(Q)=\int_{Q}w(x)dx,$
\[
\int_{0}^{t}[K(s,g\chi_{Q};L_{w}^{1}(Q),L^{\infty}(Q))s^{-\theta}]^{p}%
\frac{ds}{s}\leq C[K(s,g\chi_{Q};L_{w}^{1}(Q),L^{\infty}(Q))s^{-\theta}]^{p}.
\]
In order to avoid repetitions we shall leave further details for the
interested reader and, in particular, refer to \cite{MarM} where also the
$K-$functional for the pair $(L_{w}^{1}\left(  Q\right)  ,L^{\infty}\left(
Q\right)  )$ is computed.

\section{Applications and comparison with results in the literature
\label{sec:Comparison}}

\subsection{Reverse H\"{o}lder inequalities in the setting of Lorentz
spaces\label{sec:lorentz}}

In this section we consider the class of weights $RH_{L(p,q)}$ that satisfy
$L(p,q)$ reverse H\"{o}lder's inequalities. Our main result here can be
summarized as follows%
\[
RH_{L(p,q)}=RH_{p},1<p\leq q.
\]

Let us recall some definitions.

\begin{definition}
Let $1<p<\infty,1\leq q<\infty.$We let $L\left(  p,q\right)  =\{f:\left\Vert
f\right\Vert _{L\left(  p,q\right)  }=\left(  \int_{0}^{\infty}f^{\ast\ast
}\left(  t\right)  ^{q}t^{\frac{q}{p}}\frac{dt}{t}\right)  ^{\frac{1}{q}%
}<\infty\}$.
\end{definition}

\begin{definition}
Let $1<p<\infty,1\leq q<\infty.$ We shall say that $w$ satisfies an $L(p,q)$
reverse H\"{o}lder Lorentz inequality$,$ and we shall write $w\in
RH_{L(p,q)},$ if and only if there exists a constant $C>0$ such that, for all
cubes $Q,$ we have%
\begin{equation}
\frac{\left\Vert w\chi_{Q}\right\Vert _{L(p,q)}}{\left\vert Q\right\vert
^{1/p}}\leq\frac{C}{\left\vert Q\right\vert }\int_{Q}w(x)dx. \label{p,q}%
\end{equation}

We let
\end{definition}

\[
\left\Vert w\right\Vert _{RH_{L(p,q)}}=\inf\{C:\text{(\ref{p,q}) holds}\}.
\]

\begin{theorem}
Let $1<p<\infty,p\leq q<\infty.$ Then,
\[
RH_{L(p,q)}=RH_{p}.
\]

\end{theorem}

\begin{proof}
The containment $RH_{p}\subset RH_{L(p,q)}$ is automatic since $L^{p}\subset
L(p,q).$ Suppose that $w\in$ $RH_{L(p,q)}.$ Fix a cube $Q_{0}.$ Applying
(\ref{p,q}) to $w\chi_{Q_{0}}$, we get
\begin{equation}
M_{L(p,q),Q_{0}}(w\chi_{Q_{0}})(x):=\sup\limits_{Q_{0}\supset Q\ni x}%
\frac{\left\Vert (w\chi_{Q_{0}}\chi_{Q})\right\Vert _{L(p,q)}}{\left\vert
Q\right\vert ^{1/p}}\leq C\left\Vert w\right\Vert _{RH_{L(p,q)}}M(w\chi
_{Q_{0}})(x). \label{(p,q)2}%
\end{equation}
Then, taking rearrangements in (\ref{(p,q)2}), and using the familiar Herz
inequality to estimate the right hand side, yields
\[
(M_{L(p,q),Q_{0}}(w\chi_{Q_{0}}))^{\ast}(t)\leq\frac{C\left\Vert w\right\Vert
_{RH_{L(p,q)}}}{t}\int_{0}^{t}(w\chi_{Q_{0}})^{\ast}(s)ds,0<t<\left\vert
Q_{0}\right\vert .
\]
The left hand side can be estimated by the local version of an estimate
obtained in \cite[Corollary (i), page 69]{BMR2},%
\[
\frac{1}{t^{1/p}}\int_{0}^{t}[(w\chi_{Q_{0}})^{\ast}(s)s^{1/p}]^{q}\frac
{ds}{s}\leq c(M_{L(p,q),Q_{0}}w(x))^{\ast}(t),0<t<\left\vert Q_{0}\right\vert
.
\]
Combining these estimates we thus find that there exists a constant $C>0$,
such that for all $t>0,$%
\begin{equation}
\frac{1}{t^{1/p}}\int_{0}^{t}(w\chi_{Q_{0}})^{\ast}(s)^{q}s^{q/p}\frac{ds}%
{s}\leq C\frac{\left\Vert w\right\Vert _{RH_{L(p,q)}}}{t}\int_{0}^{t}%
(w\chi_{Q_{0}})^{\ast}(s)ds. \label{leida}%
\end{equation}
Now, recall that (cf. \cite{BL}, \cite{torchinsky})
\[
K(t,w;L(p,q),L^{\infty})\approx\left\{  \int_{0}^{t^{p}}[w^{\ast}%
(s)s^{1/p}]^{q}\frac{ds}{s}\right\}  ^{1/q},K(t,w;L^{1},L^{\infty})=\int%
_{0}^{t}w^{\ast}(s)ds.
\]
Let $\theta=1-1/p=1/p^{\prime},$ using Holmstedt's formula in a familiar way
we can rewrite (\ref{leida}) as%
\[
\frac{1}{t^{1-\theta}}K(t^{1-\theta},w\chi_{Q_{0}};(L^{1},L^{\infty}%
)_{\theta,q},L^{\infty})\leq C\frac{\left\Vert w\right\Vert _{RH_{L(p,q)}}}%
{t}K(t,w\chi_{Q_{0}};L^{1},L^{\infty}),
\]
\bigskip therefore, since $Q_{0}$ was arbitrary, a simple change of variables
shows that $w$ satisfies (\ref{Ktheta}). More precisely, we see that for all
cubes $Q,w\chi_{Q}\in RH_{1/p^{\prime},q}(L^{1}(Q),L^{\infty}(Q)),$ with%
\[
\sup_{Q}\left\Vert w\chi_{Q}\right\Vert _{RH_{1/p^{\prime},q}(L^{1}%
(Q),L^{\infty}(Q))}\preceq\left\Vert w\right\Vert _{RH_{L(p,q)}}.
\]
Consequently, applying Theorem \ref{teo:p}, we conclude that $w\in RH_{p}.$
\end{proof}

\begin{corollary}
Let $1<p\leq q<\infty.$ Then, $w\in$ $RH_{L(p,q)}$ if and only if there exists
a constant $C>0$ such that, for all cubes $Q,$ $0<t<\left\vert Q\right\vert ,$
we have
\[
\int_{0}^{t}[K(s,w\chi_{Q};L^{1}(Q),L^{\infty}(Q))s^{-1/p^{\prime}}]^{q}%
\frac{ds}{s}\leq C[t^{-1/p^{\prime}}K(s,w\chi_{Q};L^{1}(Q),L^{\infty}%
(Q))]^{q}.
\]

\end{corollary}

\subsection{Comparisons with recent results on $A_{\infty}$\label{sub:carro}}

In the recent paper \cite{ACKS}, which is apparently independent of the
literature on indices or the interpolation methods, as discussed in this
paper, the authors defined an index on weights which, among other interesting
applications, was used to characterize the Muckenhoupt class of $A_{\infty}$
weights. The purpose of this section is to compare the results of \cite{ACKS}
with ours.

Let us start by recalling the notion of index defined in \cite{ACKS}. We shall
say a weight $w$ has finite index, in the sense of \cite{ACKS}, if there exist
$r\in\left(  0,1\right)  $ and $\lambda,\widetilde{\gamma}>0$ such that for
all cubes $Q$, $0<s\leq t<r\left\vert Q\right\vert ,$ it holds%
\begin{equation}
\frac{(w\chi_{Q})^{\ast}\left(  s\right)  }{(w\chi_{Q})^{\ast}\left(
t\right)  }\leq\widetilde{\gamma}\left(  \frac{s}{t}\right)  ^{-\lambda}.
\label{acks1}%
\end{equation}
In \cite{ACKS} the authors then let%
\[
\widetilde{ind}(w)=\inf\{\lambda:\text{(\ref{acks1}) holds}\},
\]
and show that%
\[
A_{\infty}=\{w:\widetilde{ind}(w)<1\}.
\]
For comparison we note that since%
\[
A_{\infty}=RH
\]
from Theorem \ref{teo:galvanizado} (ii) we obtain%
\[
A_{\infty}=\{w:ind\{K(\cdot,w\chi_{Q},L^{1}(Q),L^{\infty}(Q))\}_{Q}>0\}.
\]
In this section we compare and clarify these results by means of a direct
proof of

\begin{theorem}
\label{teo:delossuecos}
\begin{equation}
ind\{K(\cdot,w\chi_{Q},L^{1}(Q),L^{\infty}(Q))\}_{Q}>0\Leftrightarrow
\widetilde{ind}(w)<1. \label{acks2}%
\end{equation}

\end{theorem}

\begin{remark}
\label{lonsome}Before going through the proof of (\ref{acks2}), let us observe
that we can rewrite the condition (\ref{acks1}) as follows: for all cubes $Q$,
$0<s\leq t<r\left\vert Q\right\vert ,$ it holds that
\[
s(w\chi_{Q})^{\ast}\left(  s\right)  s^{-(1-\lambda)}\leq\widetilde{\gamma
}t(w\chi_{Q})^{\ast}\left(  t\right)  t^{-(1-\lambda)}.
\]
In other words, the function $x(w\chi_{Q})^{\ast}\left(  x\right)
x^{-(1-\lambda)}$ is a.i. on $(0,r\left\vert Q\right\vert ),$ with constant of
a.i. $\widetilde{\gamma}$ independent of $Q.$ Therefore we readily see that%
\begin{align*}
\widetilde{ind}(w) &  =\sup_{\delta>0}\{\delta>0:\exists\widetilde{\gamma
}>0,r\in(0,1)\text{ such that for all cubes }Q,\text{ }x(w\chi_{Q})^{\ast
}\left(  x\right)  x^{-\delta}\text{ }\\
&  \text{is a.i. on }(0,r\left\vert Q\right\vert )\text{ with }%
\widetilde{\gamma}\text{ constant of a.i.}\}.
\end{align*}
For comparison, in Definition \ref{def:estrela} our index, $ind,$ was defined
using functions built around the family of functions $\phi_{w,Q}%
(t)=K(t,w\chi_{Q},L^{1}(Q),L^{\infty}(Q))=\int_{0}^{t}(w\chi_{Q})^{\ast
}(s)ds,$ and we made some (minimal) use of the fact that $\phi_{w,Q}(t)$
increases and $\frac{\phi_{w,Q}(t)}{t}$ decreases. But if we formally apply
our definition to the family of functions $\{t\phi_{w,Q}^{\prime}%
(t)\}_{Q}=\{t$ $(w\chi_{Q})^{\ast}(t)\}_{Q}$ we see that formally we have
\[
\widetilde{ind}(w)=ind\{t\phi_{w,Q}^{\prime}(t)\}_{Q}\}=ind\{t(w\chi
_{Q})^{\ast}(t)\}_{Q}.
\]

\end{remark}

We are now ready for the \textbf{Proof of Theorem \ref{teo:delossuecos}}

\begin{proof}
Suppose that $\widetilde{ind}(w)<1.$ Then there exists $\exists\delta
>0,\gamma\in(0,1)$ such that for all cubes, $t\phi_{w,Q}^{\prime}%
(t)t^{-\delta}$ is a.i. on $(0,\gamma\left\vert Q\right\vert ).$ It follows
that for any $0<t<h<\gamma\left\vert Q\right\vert ,$ we have%
\begin{align*}
\phi_{w,Q}(t)t^{-\delta}  &  =t^{-\delta}\int_{0}^{t}\phi_{w,Q}^{\prime
}(s)s^{1-\delta}s^{\delta-1}ds\\
&  \leq t^{-\delta}\phi_{w,Q}^{\prime}(t)t^{1-\delta}\frac{t^{\delta}}{\delta
}\\
&  \leq\frac{C}{\delta}\phi_{w,Q}^{\prime}(h)h^{1-\delta}\\
&  =\frac{C}{\delta}h^{-\delta}\left(  \phi_{w,Q}^{\prime}(h)h\right) \\
&  \leq\frac{C}{\delta}h^{-\delta}\int_{0}^{h}\phi_{w,Q}^{\prime}(r)dr\\
&  =\frac{C}{\delta}h^{-\delta}\phi_{w,Q}(h).
\end{align*}
Therefore for all cubes $Q,$ $t^{-\delta}\phi_{w,Q}(t)$ is a.i. on
$(0,\gamma\left\vert Q\right\vert ).$ Thus, $ind\{\phi_{w,Q}\}_{Q}>0.$

Conversely, if $ind\{\phi_{w,Q}\}_{Q}>0,$ then by Theorem
\ref{teo:galvanizado} and Lemma \ref{benson} there exists $p>1,$ $C>0,$ such
that for $t\in\left(  0,\left\vert Q\right\vert \right)  $%
\[
\left\{  \frac{1}{t}\int_{0}^{t}[(w\chi_{Q})^{\ast\ast}(s)]^{p}ds\right\}
^{1/p}\leq C\left\{  \frac{1}{t}\int_{0}^{t}(w\chi_{Q})^{\ast}(s)ds\right\}
,
\]
which implies%
\begin{equation}
\left\{  \int_{0}^{t}[(w\chi_{Q})^{\ast}(s)]^{p}ds\right\}  ^{1/p}\leq
C\left\{  t^{-1/p^{\prime}}\int_{0}^{t}(w\chi_{Q})^{\ast}(s)ds\right\}  .
\label{lu}%
\end{equation}
Let $\rho>1$ be a number will be chosen precisely later. Then, for
$t\in\left(  0,\frac{\left\vert Q\right\vert }{\rho}\right)  ,$ we have%
\begin{align*}
\int_{0}^{t}(w\chi_{Q})^{\ast}(s)ds  &  \leq\left\{  \int_{0}^{t}[(w\chi
_{Q})^{\ast}(s)]^{p}ds\right\}  ^{1/p}t^{1/p^{\prime}}\\
&  \leq\left\{  \int_{0}^{t\rho}[(w\chi_{Q})^{\ast}(s)]^{p}ds\right\}
^{1/p}t^{1/p^{\prime}}\\
&  \leq C\left\{  (t\rho)^{-1/p^{\prime}}\int_{0}^{\rho t}(w\chi_{Q})^{\ast
}(s)ds\right\}  t^{1/p^{\prime}}\text{ (by (\ref{lu}))}\\
&  =C\rho^{-1/p^{\prime}}\int_{0}^{t}(w\chi_{Q})^{\ast}(s)ds+C\rho
^{-1/p^{\prime}}\int_{t}^{\rho t}(w\chi_{Q})^{\ast}(s)ds
\end{align*}
Rearranging terms, and using the fact that $(w\chi_{Q})^{\ast}$ is decreasing,
we find%
\[
(1-C\rho^{-1/p^{\prime}})\int_{0}^{t}(w\chi_{Q})^{\ast}(s)ds\leq
C\rho^{-1/p^{\prime}}(\rho-1)t(w\chi_{Q})^{\ast}(t).
\]
Therefore if we choose $\rho>1$ such that $C\rho^{-1/p^{\prime}}<1$ and use
once again the fact that $(w\chi_{Q})^{\ast}$ is decreasing, we obtain that on
$\left(  0,\frac{\left\vert Q\right\vert }{\rho}\right)  ,$
\begin{equation}
t(w\chi_{Q})^{\ast}(t)\leq\int_{0}^{t}(w\chi_{Q})^{\ast}(s)ds\leq\frac
{C\rho^{-1/p^{\prime}}(\rho-1)}{(1-C\rho^{-1/p^{\prime}})}t(w\chi_{Q})^{\ast
}(t). \label{demandada}%
\end{equation}
Now, since $ind\{\phi_{w,Q}\}_{Q}>0$, there exists $\delta>0$,$\gamma\in(0,1)$
such that $t^{-\delta}\int_{0}^{t}(w\chi_{Q})^{\ast}(s)ds$ is a.i. on
$(0,\gamma\left\vert Q\right\vert ).$ Consequently, if we further demand that
$\rho>1/\gamma,$ we see that (\ref{demandada}) implies that $t^{-\delta
}t(w\chi_{Q})^{\ast}(t)$ is a.i. and therefore by Remark \ref{lonsome} we find
that
\[
\widetilde{ind}(w)<1,
\]
as we wished to show.
\end{proof}

\begin{remark}
In retrospect it is interesting to observe that while the theory of
\cite{ACKS} was apparently developed independently from theory of indices, and
interpolation theory, one of the first results obtained in \cite{ACKS} is the
control of integrals of the form $\int_{0}^{t}f\left(  t\right)  dt$ by
$tf\left(  t\right)  ,$ where $f$ is decreasing.
\end{remark}

\subsection{An interpolation theorem involving extrapolation spaces: Operators
acting on $RH$ weights.\label{extrapola}}

Another interesting way of characterizing $A_{\infty},$ apparently first given
by Fujii (cf. \cite{F}, \cite{DMO} and the references therein), can be stated
as follows: $w$ belongs to the $A_{\infty}$ class if and only if there exists
a constant $C$ such that for all cubes $Q,$%
\begin{equation}
\int_{Q}M(w\chi_{Q})(x)dx\leq C\int_{Q}w(x)dx.\label{re2}%
\end{equation}

We investigate the connection of (\ref{re2}) with our own characterization of
$A_{\infty}$ using interpolation. More precisely, in this section we prove an
abstract interpolation theorem modelled after a result obtained in
\cite{ACKS}\footnote{On closer examination one can see that the result is
closely related to an extrapolation version of a theorem due Zygmund (cf.
\cite{as}, \cite{Go}, \cite{JM91} and the discussion in Remark \ref{re:jm}
below).}, that when applied to the maximal operator shows that if $w\in RH$
then $w$ satisfies the Fujii condition (\ref{re2}).

Let $\vec{X}$ be an ordered Banach pair. We recall the definition of
$\left\Vert \cdot\right\Vert _{RH_{0,1}(\vec{X})}$ that we introduced in
Definition \ref{interpola},%
\begin{equation}
\left\Vert w\right\Vert _{RH_{0,1}(\vec{X})}=\inf\{c:\int_{0}^{t}%
K(s,w;;\vec{X})\frac{ds}{s}\leq cK(t,w;\vec{X})\}. \label{tres}%
\end{equation}

Let us also recall the definition of the notion of \textquotedblleft
generalized weak types $(1,1),(\infty,\infty)^{\prime\prime}$ as given in
\cite{De}. We shall say that $T$ is of generalized weak types $(1,1),(\infty
,\infty),$ if there exists a constant $C>0$ such that%
\begin{equation}
\frac{K(r,Tf;\vec{X})}{r}\leq C\{\frac{1}{r}\int_{0}^{r}K(s,f;\vec{X}%
)\frac{ds}{s}+\int_{r}^{\infty}\frac{K(s,f;\vec{X})}{s}\frac{ds}{s}\},r>0.
\label{two}%
\end{equation}

\begin{theorem}
\label{loextrapola}Let $\vec{X}$ be an ordered Banach pair, and let $n$ be the
norm of the embedding $X_{1}\subset X_{0}.$ Let $T$ be an operator of
generalized weak types $(1,1),(\infty,\infty).$ Then, there exists an absolute
constant $c>0,$ such that%
\begin{equation}
\int_{0}^{t}\frac{K(r,Tw;\vec{X})}{r}dr\leq c\left\Vert w\right\Vert _{A_{0}%
}(\left\Vert w\right\Vert _{RH_{0,1}(\vec{X})}^{2}+\left\Vert w\right\Vert
_{RH_{0,1}(\vec{X})}+1),0<t<n, \label{cinco}%
\end{equation}%
\begin{equation}
\left\Vert Tf\right\Vert _{\vec{X}_{0,1}}\leq c\left\Vert w\right\Vert
_{A_{0}}(\left\Vert w\right\Vert _{RH_{0,1}(\vec{X})}^{2}+1+\left\Vert
w\right\Vert _{RH_{0,1}(\vec{X})}). \label{cuatro}%
\end{equation}

\end{theorem}

\begin{proof}
Let $w\in RH_{0,1}(\vec{X}).$ Integrating (\ref{two}) we obtain,%
\begin{align*}
\int_{0}^{t}\frac{K(r,Tf;\vec{X})}{r}dr  &  \preceq\int_{0}^{t}\frac{1}{r}%
\int_{0}^{r}K(s,w;\vec{X})\frac{ds}{s}dr+\int_{0}^{t}\int_{r}^{\infty}%
\frac{K(s,w;\vec{X})}{s}\frac{ds}{s}dr\\
&  =(I)+(II).
\end{align*}
Using (\ref{tres}) (twice) we find%
\begin{align*}
(I)  &  \preceq\int_{0}^{t}\frac{1}{r}\left\Vert w\right\Vert _{RH_{0,1}%
(\vec{X})}K(r,w;\vec{X})dr\\
&  \preceq\left\Vert w\right\Vert _{RH_{0,1}(\vec{X})}^{2}K(t,w;\vec{X})\\
&  \preceq\left\Vert w\right\Vert _{RH_{0,1}(\vec{X})}^{2}\left\Vert
w\right\Vert _{A_{0}}.
\end{align*}
To estimate $(II)$ we integrate by parts. For this purpose note that
$K(s,w;\vec{X})\leq\lim_{s\rightarrow\infty}K(s,w;\vec{X})=\left\Vert
w\right\Vert _{A_{0}},$ and therefore
\[
\lim_{r\rightarrow0}(r\int_{r}^{\infty}\frac{K(s,w;\vec{X})}{s}\frac{ds}%
{s})\leq\lim_{r\rightarrow0}(r\left\Vert w\right\Vert _{A_{0}}\int_{r}%
^{\infty}\frac{ds}{s^{2}})=\left\Vert w\right\Vert _{A_{0}}.
\]
Consequently, we get
\begin{align*}
(II)  &  \preceq\left.  r\int_{r}^{\infty}\frac{K(s,w;\vec{X})}{s}\frac{ds}%
{s}\right\vert _{0}^{t}+\int_{0}^{t}\frac{K(r,w;\vec{X})}{r}dr\\
&  \preceq t\int_{t}^{\infty}\frac{K(s,w;\vec{X})}{s}\frac{ds}{s}+\left\Vert
w\right\Vert _{RH_{0,1}(\vec{X})}K(t,w;\vec{X})\\
&  \preceq\left\Vert w\right\Vert _{A_{0}}+\left\Vert w\right\Vert
_{RH_{0,1}(\vec{X})}\left\Vert w\right\Vert _{A_{0}}.
\end{align*}
Combining estimates yields,%
\begin{equation}
\int_{0}^{t}\frac{K(r,Tw;\vec{X})}{r}dr\preceq\left\Vert w\right\Vert _{A_{0}%
}(\left\Vert w\right\Vert _{RH_{0,1}(\vec{X})}^{2}+\left\Vert w\right\Vert
_{RH_{0,1}(\vec{X})}+1). \label{laT}%
\end{equation}
Letting $t\rightarrow n$ we then obtain (\ref{cuatro}).
\end{proof}

We apply this result to the family of pairs $\vec{X}=(L^{1}(Q),L^{\infty
}(Q)),$ and the maximal operator $M$. Indeed, as it is well known, the maximal
operator $M$ satisfies (\ref{two}). For the benefit of the reader we offer a
quick verification here using the familiar Herz's equivalence. Indeed, we
have
\[
\frac{K(r,Mf,\vec{X})}{r}=(Mf)^{\ast\ast}(r)=\frac{1}{r}\int_{0}^{r}%
(Mf)^{\ast}(s)ds\approx\frac{1}{r}\int_{0}^{r}f^{\ast\ast}(s)ds=\frac{1}%
{r}\int_{0}^{r}K(s,f,\vec{X})\frac{ds}{s}.
\]
Suppose now that $w\in RH_{LLogL},$ then, by Theorem \ref{teo:limite}, we have
that, for all cubes $Q,$ $w\in RH_{0,1}(L^{1}(Q),L^{\infty}(Q)),$ and
\[
\sup_{Q}\left\Vert w\right\Vert _{RH_{0,1}(L^{1}(Q),L^{\infty}(Q))}%
\approx\left\Vert w\right\Vert _{RH_{LLogL}}.
\]
If we apply (\ref{laT}) with $t=\left\vert Q\right\vert ,$ we have%
\[
\int_{0}^{\left\vert Q\right\vert }\frac{K(r,M(w\chi_{Q});\vec{X})}%
{r}dr\preceq\left\Vert w\right\Vert _{L^{1}(Q)}(\left\Vert w\right\Vert
_{RH_{LLogL}}^{2}+\left\Vert w\right\Vert _{RH_{LLogL}}+1).
\]
Now, we observe that%
\begin{align}
\int_{Q}M(w\chi_{Q})(x)dx  &  =\int_{0}^{\left\vert Q\right\vert }[M(w\chi
_{Q})]^{\ast}(r)dr\nonumber\\
&  \leq\int_{0}^{\left\vert Q\right\vert }\frac{K(r,M(w\chi_{Q});\vec{X})}%
{r}dr.\nonumber
\end{align}
Combining these estimates we obtain
\begin{align}
\int_{Q}M(w\chi_{Q})(x)dx  &  \preceq\left\Vert w\right\Vert _{L^{1}%
(Q)}(\left\Vert w\right\Vert _{RH_{LLogL}}^{2}+\left\Vert w\right\Vert
_{RH_{LLogL}}+1)\nonumber\\
&  =\int_{Q}w(x)dx(\left\Vert w\right\Vert _{RH_{LLogL}}^{2}+\left\Vert
w\right\Vert _{RH_{LLogL}}+1) \label{fuji}%
\end{align}

This shows that our characterization of $RH_{LLogL}=RH=A_{\infty},$ implies
Fujii's condition.

In the next remark we show directly how the condition (\ref{re2}) implies the
defining condition of $RH_{LLogL}.$

\begin{remark}
\label{sec:fuji}Suppose that $w$ satisfies Fujii's condition (\ref{re2}). Let
$x\in R^{n},$ then for any cube $x\in Q$ we have (cf. the argument in
\cite[pag 174]{Perez})%
\begin{align*}
\frac{1}{\left\vert Q\right\vert }\int_{Q}M(w)(y)dy &  \leq\frac{1}{\left\vert
Q\right\vert }\int_{Q}M(w\chi_{3Q})(y)dy+\frac{1}{\left\vert Q\right\vert
}\int_{Q}M(w(1-\chi_{3Q}))(y)dy\\
&  \leq\frac{c}{\left\vert 3Q\right\vert }\int_{3Q}M(w\chi_{3Q})(y)dy+c\inf
_{z\in Q}Mw(z)\\
&  \leq\frac{\tilde{c}}{\left\vert 3Q\right\vert }\int_{3Q}w(x)dx+cMw(x)\text{
(by (\ref{re2}))}\\
&  \leq CMw(x).
\end{align*}
Consequently, for all $x\in R^{n},$%
\[
M(Mw)(x)\leq CMw(x).
\]
Taking rearrangements, and then applying Herz's inequality, yields%
\begin{equation}
\left(  M(Mw)\right)  ^{\ast}(t)\preceq\left(  Mw\right)  ^{\ast
}(t).\label{porlos dos}%
\end{equation}
Applying Herz's equivalence (repeatedly), and the known calculation of the
corresponding $K-$functional, we get the following equivalent expressions to
the left and right hand sides of (\ref{porlos dos}):
\[
\left(  M(Mw)\right)  ^{\ast}(t)\approx\frac{1}{t}\int_{0}^{t}(Mw)^{\ast
}(s)ds\approx\frac{1}{t}\int_{0}^{t}f^{\ast\ast}(s)ds=\frac{1}{t}\int_{0}%
^{t}K(s,f;L^{1},L^{\infty})\frac{ds}{s},
\]%
\[
\left(  Mw\right)  ^{\ast}(t)\approx\frac{1}{t}\int_{0}^{t}f^{\ast
}(s)ds=K(t,f;L^{1},L^{\infty}).
\]
Rewriting (\ref{porlos dos}) using this information we get%
\[
\int_{0}^{t}K(s,f;L^{1},L^{\infty})\frac{ds}{s}\preceq K(t,f;L^{1},L^{\infty
}).
\]
Thus, if $w$ is a weight satisfies (\ref{re2}) then $w\in RH_{LLogL}.$
\end{remark}

\begin{remark}
\label{re:jm}We cannot resist to point out a connection to extrapolation
theory. Indeed, the theory of \cite{JM91} produces many examples of operators
weak types $(1,1),(\infty,\infty)$ by means of extrapolating inequalities. A
prototype result can be stated as follows: if $T$ is an operator on a real
interpolation scale $\{\vec{X}_{\theta,q}\}_{\theta,q},$ $\theta\in(0,1),$
$q\geq1,$ such that%
\[
\left\Vert T\right\Vert _{\vec{X}_{\theta,q}\rightarrow\vec{X}_{\theta,q}%
}\preceq(1-\theta)^{-1}\theta^{-1}%
\]
then $T$ satisfies (\ref{re2}) (cf. \cite{JM91}). In particular, if $\vec{X}$
is ordered and we are interested only on the behavior on spaces near the
larger space $X_{0},$ that is when $\theta\rightarrow0,$ then operators that
satisfy $\left\Vert T\right\Vert _{\vec{X}_{\theta,q}\rightarrow\vec
{X}_{\theta,q}}\preceq\theta^{-1}$ as $\theta\rightarrow0,$ can be
characterized by (cf. \cite{JM91})%
\begin{equation}
\frac{K(r,Tf;\vec{X})}{r}\leq C\{\frac{1}{r}\int_{0}^{r}K(s,f;\vec{X}%
)\frac{ds}{s}\}\label{loextrapola2}%
\end{equation}
In particular, when $\vec{X}=(L^{1},L^{\infty})$ this leads to the
rearrangement inequalities,%
\begin{equation}
\frac{1}{r}\int_{0}^{r}(Tf)^{\ast}(s)ds\leq\frac{C}{r}\int_{0}^{r}f^{\ast\ast
}(s)ds.\label{loextrapola1}%
\end{equation}
The idea behind Theorem \ref{loextrapola} is that if $f\in RH_{LLogL}$ we
have
\[
\frac{1}{r}\int_{0}^{r}K(s,f;\vec{X})\frac{ds}{s}\leq\left\Vert f\right\Vert
_{RH_{LLogL}}\frac{K(r,f;\vec{X})}{r}%
\]
therefore if we integrate (\ref{loextrapola1}) we can use the $RH_{LLogL}$
condition twice on the right hand side to obtain
\[
\int_{0}^{t}\frac{K(r,Tf;L^{1},L^{\infty})}{r}dr\leq C\left\Vert f\right\Vert
_{RH_{LLogL}}^{2}K(t,f;L^{1},L^{\infty})
\]
which effectively reverses (\ref{loextrapola2})!
\end{remark}

\subsection{The Stromberg-Wheeden Theorem\label{sec:strom}}

In this section we apply our theory to give a simple proof of the
Stromberg-Wheeden theorem, which is arguably one of the cornerstones of the
classical theory of weighted norm inequalities (cf. \cite{CN}).

\begin{theorem}
$w\in RH_{p}$ if and only if $w^{p}\in A_{\infty}$
\end{theorem}

\begin{proof}
Suppose that $w^{p}\in A_{\infty}$. By the characterization of $A_{\infty}$
given in \ref{teo:galvanizado}, $w^{p}\in A_{\infty}$ if and only if
\begin{equation}
ind\{K(t,w^{p}\chi_{Q};L^{1}(Q),L^{\infty}(Q))\}_{Q}>0. \label{prop0}%
\end{equation}
It follows that there exists $\delta>0$ such that $\frac{K(t,w^{p}\chi
_{Q};L^{1}(Q),L^{\infty}(Q))}{t^{\delta}}$ a.i. on $(0,\gamma\left\vert
Q\right\vert ),$ but then the proof of Theorem \ref{teo:delossuecos} shows
that we $[(w\chi_{Q})^{\ast}(s)]^{p}s^{1-\delta}$ is a.i.; therefore, raising
this function to the power $1/p,$ yields that $(w\chi_{Q})^{\ast}%
(s)s^{\frac{1-\delta}{p}}$ is also a.i. Consequently, if we let $\mu
=1-1/p+\delta/p,$ then we can write $\frac{1-\delta}{p}=1-\mu,$ and once again
by the proof of Theorem \ref{teo:delossuecos} we find that $K(s,w\chi
_{Q};L^{1}(Q),L^{\infty}(Q))s^{-\mu}=$ $K(s,w\chi_{Q};L^{1}(Q),L^{\infty
}(Q))s^{-1/p^{\prime}-\delta/p}$ is a.i., whence%
\begin{equation}
ind\{K(\cdot,w\chi_{Q};L^{1}(Q),L^{\infty}(Q))\}_{Q}>1/p^{\prime}.
\label{prop}%
\end{equation}
Consequently, by Theorem \ref{teo:galvanizado}%
\[
w\in RH_{p}.
\]
Conversely, all the steps can be reversed. Indeed, suppose that $w\in RH_{p}.$
Then (\ref{prop}) holds, and consequently for some $\delta>0,$ $K(s,w\chi
_{Q};L^{1}(Q),L^{\infty}(Q)))s^{-1/p^{\prime}-\delta/p}$ is a.i. Now the proof
of Theorem \ref{teo:delossuecos} implies that $(w\chi_{Q})^{\ast}%
(s)s^{\frac{1-\delta}{p}}$ is a.i. and hence $[(w\chi_{Q})^{\ast}%
(s)]^{p}s^{1-\delta}$ is a.i., and once again by the proof of Theorem
\ref{teo:delossuecos}, $\frac{K(t,w^{p}\chi_{Q};L^{1}(Q),L^{\infty}%
(Q))}{t^{\delta}}$ is a.i. and therefore (\ref{prop0}) holds, yielding that
$w^{p}\in A_{\infty}.$
\end{proof}

\section{Some Problems\label{sec:problems}}

We would like to close this paper with some open-ended problems connected with
the developments in this paper that we consider of some potential interest.
The problems are thus mainly focussed on exploring the connections between
weighted norm inequalities and interpolation/extrapolation theoretical methods.

\begin{enumerate}
\item \textbf{Interpolation/Extrapolation Methods}: So far, the interpolation
methods\footnote{as opposed to the more common application of interpolation
\textquotedblleft Then by interpolation".} we have been developing to study
classes of weights are built on the real method of interpolation. It is likely
that other methods of interpolation could also be of interest in this area. In
particular, the r\^{o}le of the complex method of interpolation of
Calder\'{o}n ought to be explored. For example, the Calder\'{o}n method of
interpolation of lattices, e.g. the \textquotedblleft$X_{0}^{1-\theta}%
X_{1}^{\theta}$\textquotedblright\ method, is likely to be relevant in
connection with factorizations of weighted norm inequalities and of their
underlying classes of weights. Also intriguing are the possible connections
with the interpolation method of \cite{cw}, which allows to treat the real and
complex methods of interpolation in a unified way. In particular, \cite{cw}
introduced a new variant of the $K-$functional that makes this tool available
for the complex method of interpolation. We think it could be of interest to
explore its application within the framework developed in this paper.
Likewise, the method of orbits (cf. \cite{ov}) could also play a r\^{o}le. In
fact, some results that connect the method of orbits and the abstract Gehring
Lemma was started to be explored in \cite{BMR1}, and there is a detailed
application of orbital methods to the study of self-improving (or
\textquotedblleft open" properties) in \cite{kru} (cf. also item \textbf{9}
below). It seems to us that the theory of indices is likely to have an impact
reformulating and clarifying results of \cite{kru} and its applications to the
theory of weighted norm inequalities and pde's. Likewise, the results of
Section \ref{extrapola}, and, in particular, Remark \ref{re:jm}, suggest new
potential applications of extrapolation theory to weighted norm
inequalities.\bigskip
\end{enumerate}

\begin{enumerate}
\item[2] \textbf{Function Spaces}: In connection with the set of problems
described in \textbf{1} above, it would be of interest to study classes of
weights that are naturally associated with more general function spaces. In
this direction, applications to weighted norm inequalities in the setting of
Orlicz spaces would be of obvious interest and a number of results in this
direction already appear in \cite{BMR}, \cite{Ha}, \cite{MM}, \cite{Mo1} and
the references therein. In this connection see also item 5 below.\bigskip
\end{enumerate}

\begin{enumerate}
\item[3] \textbf{Other classes of weights:} Interpolation methods have the
potential to be useful to study other classes of weights. Among the classes of
weights awaiting interpolation treatment: The class of $A_{p}$ weights (cf.
\cite{GR}), the classes $C_{p},$ $B_{p}$ of weights (cf. \cite{BMR3},
\cite{Duo}, \cite{Fer} and the references therein). Similar questions for two
weight type inequalities (cf. \cite{a}).\bigskip
\end{enumerate}

\begin{enumerate}
\item[4] \textbf{The role of constants}: in the theory of $RH_{LLogL}$ weights
is discussed, for example, in \cite{Moscariello}, \cite{bes}. We ask for a
treatment of the role of constants in the theory of weights in the context of
the interpolation/extrapolation methods. See also the next item.\bigskip

\item[5] \textbf{Reverse Hardy Inequalities.} There is an interesting
connection between Gehring's Lemma and sharp constants for reverse Hardy
inequalities (cf. \cite{BMR1}, \cite{MM}, \cite{Mii}), which we did not
discuss in this paper in order not to exasperate the editors of this volume.
In this connection it would be interesting to extend the sharp reversed Hardy
inequalities for decreasing functions, which are known for $L^{p}$ norms, to
more general function spaces. \bigskip

\item[6] \textbf{The class }$A_{p,q}$\textbf{ }of weights for which the
maximal operator of Hardy-Littlewood is bounded on $L(p,q)$ were studied in
\cite{chu}. In particular, it was shown there that $A_{p,q}=A_{p}%
,1<p<\infty,1<q<\infty.$ The case $q=1$ is discussed in \cite{hunt}. We are
not aware of a systematic study.\bigskip

\item[7] \textbf{The class }$RH_{\infty}$ was apparently first systematically
studied in \cite{CN}, where it was defined through the use of the minimal
operator%
\[
\mathfrak{M}f\left(  x\right)  =\inf\limits_{Q\ni x}\frac{1}{\left\vert
Q\right\vert }\int_{Q}\left\vert f\left(  x\right)  \right\vert dx.
\]
We say that a weight $w\in RH_{\infty}$ if there exists $C>0$ such that
$w\left(  x\right)  \leq C\mathfrak{M}w\left(  x\right)  $ $a.e.$ It would be
interesting to understand the connection of this operator with interpolation
theory.\bigskip

\item[8] \textbf{Discrete Gehring type inequalities via interpolation.} We
would like to suggest the project of understanding recent results on discrete
reverse H\"{o}lder inequalities (cf. \cite{Sak1}, see also \cite{Sak} for
related work) using the methods developed in this paper. Likewise, another
potential application of our theory is the setting of \textbf{metric spaces}
(cf. \cite{Mo1} and the references therein).\bigskip

\item[9] \textbf{Self-improving inequalities and PDEs.} This topic is of
course of central interest and continues to be a source of problems and
inspiration for applications of interpolation methods. Here is a far from
complete sample of references in this direction that we happen to be aware of:
\cite{AM}, \cite{av}, \cite{Ber}, \cite{cheng}, \cite{FeKPi}, \cite{Kenig},
\cite{sb1}, \cite{Sc}, \cite{tolsa}, and the references therein.
\end{enumerate}

\end{document}